\documentclass[11pt,twoside]{preprint}
\usepackage{times}
\usepackage{mhsymb}
\usepackage{mhequ}
\usepackage{mhenvs}
\usepackage{microtype}

\def\Tr{\mop{Tr}}
\def\E{\mathbf{E}}
\def\P{\mathbf{P}}
\def\hf{{\textstyle{1\over 2}}}
\def\${|\!|\!|}
\definecolor{darkred}{rgb}{0.9,0.1,0.1}

\def\sump{{\sum_{k+ \ell = m+n}}^{\hspace{-1em}\prime}\hspace{1em}}

\begin{document}

\title{Singular perturbations to semilinear stochastic\\  heat equations}
\author{Martin Hairer}
\institute{The University of Warwick, \email{M.Hairer@Warwick.ac.uk}}

\maketitle

\begin{abstract}
We consider a class of singular perturbations to the stochastic heat equation or semilinear variations thereof.  
The interesting feature of these
perturbations is that, as the small parameter $\eps$ tends to zero, their solutions converge to the `wrong' limit, i.e.\ they
do not converge to the solution obtained by simply setting $\eps = 0$. 
A similar effect is also observed for some (formally) small stochastic perturbations of a deterministic semilinear  parabolic PDE. 

Our proofs are based on a detailed analysis of the spatially rough component of the equations, combined 
with a judicious use of Gaussian concentration inequalities.
\end{abstract}

\section{Introduction}

A well-known `folklore' statement is that if a problem ($\Phi$) is well-posed, then the solutions to every `reasonable' 
sequence ($\Phi_\eps$) of approximating problems should converge to it. This statement is of course backed by a huge number of rigorous
mathematical statements, with virtually every theorem from numerical analysis being an example. As usual, the devil lies in the details, in this case
the meaning of the word `reasonable'. However, if the problem $\Phi_0$ obtained by formally setting $\eps = 0$
in $\Phi_\eps$ makes sense and is well-posed, then one does usually expect the solutions to $\Phi_\eps$ to converge to that
of $\Phi_0$. 

One particularly famous exception to this rule is given by approximations to stochastic differential equations:
if one constructs an approximation by smoothing out the driving noise, for example by piecewise linear interpolation on a scale $\eps$,
and then solving the corresponding random ODE, then the solutions converge to the Stratonovich interpretation of the original SDE \cite{MR0183023}. 
On the other hand, if one considers the
usual Euler approximations with a step size of length $\eps$, then these will converge to the It\^o interpretation \cite{MR1214374}.
In this case however, one may argue that the original problem isn't well-posed in the classical sense, since stochastic integrals
are too irregular to be defined pathwise in an unambiguous way. Another class of exceptions to this rule arises in the
calculus of variations where it happens routinely that a sequence of `energy functionals' $\CE_n$ converges to a limit $\CE$, 
but the corresponding sequence of minimisers converges to the minimiser of some other functional $\bar \CE \neq \CE$.
It is precisely to deal with this type of problems that the notion of  $\Gamma$-convergence was introduced \cite{GammaConv}.

The aim of this article is to provide a class of examples of stochastic PDEs where a similar type of inconsistency occurs, but without
having any ambiguity in the meaning of either the approximating or the limiting model. However, similar to the case
of the approximations to an SDE, the inconsistency will be a consequence of the roughness of the solutions (but this time it is
the spatial roughness that comes into play)
and the correction term that appears in the limit can be interpreted as a kind of local quadratic variation, just as in the 
It\^o-Stratonovich correction.

The aim of this article is to study the limit $\eps \to 0$ of a class of stochastic PDEs of the type
\begin{equs}\label{e:model}\tag{$\Phi_\eps$}
\d_t u_\eps &= \nu \d_x^2 u_\eps + f(u_\eps) + \sqrt 2 \xi(t,x)\\
&\qquad  - \eps^2 \d_x^4 u_\eps + \eps g(u_\eps)\d_x^2 u_\eps + \eps h(u_\eps)\bigl(\d_x u_\eps \otimes \d_x u_\eps\bigr) 
\end{equs}
where $\xi$ is space-time white noise\footnote{Recall that space-time white noise is the distribution-valued centred Gaussian process with covariance given by $\E \xi(x,t)\xi(y,s) = \delta(x-y)\delta(t-s)$.} and $x \in [0,2\pi]$. For simplicity, we consider 
periodic boundary conditions, but we do not expect the result to depend on this.
Note that additional constants in front of the noise term $\xi$ can always be eliminated
by a time change, whereas additional constants in front of the term $\d_x^4 u$ can be eliminated by redefining $\eps$, $f$, $g$ and $h$.
Here, we will assume that the solution $u$ takes values in $\R^n$ and that
\begin{equ}
f\colon \R^n \to \R^n\;,\quad g\colon \R^n \to L(\R^n, \R^n)\;,\quad h\colon \R^n \to L(\R^n \otimes \R^n, \R^n)\;,
\end{equ}
are functions such that $f$ and $h$ are of class $\CC^1$ and $g$ is $\CC^2$. 
It is not too difficult to show \cite{DaPrato-Zabczyk92,Hairer08} that \eref{e:model} has unique local weak and mild solutions in 
$L^\infty$, so that our model is well-posed. Note furthermore that the stochastic noise term is additive, so that there is no need to resort to
stochastic calculus to make sense of the solutions to \eref{e:model}. In particular, there is no ambiguity to the concept of solution to \eref{e:model}.
\footnote{Since the solutions to \eref{e:model} are not $\CC^2$ in general (they are only $\CC^\alpha$ for every $\alpha < {3\over 2}$),
the term $\eps g(u_\eps)\d_x^2 u_\eps$ has to be interpreted in the weak sense by performing one integration by parts.}

Similarly, the model where we set $\eps = 0$, namely
\begin{equ}[e:modellimitwrong]\tag{$\Phi_0$}
\d_t u_0 = \nu \d_x^2 u_0 + f(u_0) + \sqrt 2 \xi(t,x)\;,
\end{equ}
is perfectly well-posed and has unique local solutions in $L^\infty$ both in the mild and the weak sense.

It is therefore natural to expect that one has $u_\eps \to u_0$ in some sense as $\eps \to 0$. Surprisingly, this is not the case!
Indeed, the purpose of this article is to prove a convergence result for \eref{e:model} that can loosely be formulated as:
\begin{theorem}\label{theo:main}
The solutions $u_\eps$ to \eref{e:model} converge in probability as $\eps \to 0$ to the solutions $u$ to
\begin{equ}[e:limitmain]\tag{$\Phi$}
\d_t u = \nu \d_x^2 u + \bar f(u)  + \sqrt 2\xi(t,x)\;,
\end{equ}
where the effective nonlinearity $\bar f$ is given by
\begin{equ}[e:ftilde]
\bar f(u) = f(u) + {1\over 2\sqrt \nu} \Tr \bigl(h(u) - D_sg(u)\bigr)\;.
\end{equ}
Furthermore, for every $\kappa > 0$, the convergence rate is of order $\CO(\eps^{{1\over 2}-\kappa})$ in the $L^\infty$-norm.
\end{theorem}

\begin{remark}\label{rem:notation}
The partial trace of $h$ is defined, in components, as $(\Tr h)_i = \sum_j h_{ijj}$
and the symmetrised derivative of $g$ is given by $(D_sg)_{ijk} = {1\over 2} \bigl(\d_k g_{ij} + \d_j g_{ik}\bigr)$.
\end{remark}

\begin{remark}
Note that no growth condition is assumed on $f$, $g$ and $h$. Therefore,
the limiting problem \eref{e:limitmain} could potentially have a finite explosion time $\tau^*$. In this case, the approximation 
result obviously holds only up to $\tau^*$. Furthermore, it requires to consider initial conditions that are `nice' in the sense that
they have the same structure as the solutions to \eref{e:limitmain} at positive times. A precise formulation of these conditions
can be found in the rigorous version of the statement, Theorem~\ref{theo:mainrigor} below.
\end{remark}

\begin{remark}
The precise form of the linear operator in \eref{e:model}, namely $\CL_\eps = \nu \d_x^2 - \eps^2 \d_x^4$ is not 
important. Indeed, if $P$ is any polynomial with positive leading coefficient and such that $P(0) = 1$, we could 
have replaced $\CL_\eps$ by $\nu \d_x^2 P(-\eps^2 \d_x^2)$, yielding similar results, but with the factor ${1\over 2}$
in front of the correction term in \eref{e:ftilde} replaced by some other constant.
\end{remark}

\begin{remark}
The convergence rate $\eps^{{1\over 2} - \kappa}$ is optimal in the sense that even in the `trivial' linear
case $f = g = h = 0$, the $L^\infty$-distance between $u_\eps$ and $u$ is no better than $\eps^{1\over 2}$ times some
logarithmic correction factor. 
\end{remark}

In the case $g=0$, one interpretation of our result is that the solutions to
\begin{equ}
\d_t \tilde u_\eps = \nu \d_x^2 \tilde u_\eps - \eps^2 \d_x^4 \tilde u_\eps + f(\tilde u_\eps) + \eps h(\tilde u_\eps) \bigl(\d_x \tilde u_\eps \diamond \d_x \tilde u_\eps\bigr)+ \sqrt 2 \xi(t,x)\;,
\end{equ}
do indeed converge as $\eps \to 0$ to the solutions to the corresponding system where we formally set $\eps = 0$, namely
\begin{equ}
\d_t \tilde u = \nu \d_x^2 \tilde u + f(\tilde u) + \sqrt 2 \xi(t,x)\;.
\end{equ}
In this equation, the Wick product $\diamond$ should be interpreted as the Wick product with respect to the Gaussian
structure induced on the solution space by the linearised equation, as for example in \cite{MR815192,MR1462228,MR2016604,DaPT}. 
In our situation, this means that 
\begin{equ}
\d_x \tilde u_\eps \diamond \d_x \tilde u_\eps \eqdef \d_x \tilde u_\eps \otimes \d_x \tilde u_\eps - \E \bigl(\d_x v_\eps \otimes \d_x v_\eps\bigr)\;,
\end{equ}
where $v_\eps$ is the stationary solution to the linearised equation $\d_t v_\eps = \nu \d_x^2 v_\eps - \eps^2 \d_x^4 v_\eps + \sqrt 2 \xi$.
Note that this notion of Wick product is  \textit{different} in general
from the notion obtained by using the Gaussian structure induced by the driving noise (i.e.\ the Wick product of Wiener chaos)
as in \cite{WickBook,MR1743612,Boris}. In our situation, we actually expect both notions of Wick product to lead to the same limit
as $\eps \to 0$, but we have no proof of this statement.

\subsection{Extensions and related problems}

There are a number of extensions to these results that are worth mentioning. Note first that if, instead of space-time white noise,
we took any noise process with more regular samples (for example such that its covariance operator is given by $(-\d_x^2)^{-\alpha}$ for some
$\alpha > 0$), then the solutions to \eref{e:model} would indeed converge to
those of \eref{e:modellimitwrong}. However, if $\alpha \in (0,{1\over 2})$ and one replaces the factor $\eps$ in 
front of the nonlinearity in \eref{e:model} by $\eps^{1-2\alpha}$, then one can retrace the proofs in this article to show that one
has convergence to \eref{e:limitmain}, but with the factor $1/(2\sqrt \nu)$ in \eref{e:ftilde} replaced by
\begin{equ}
{1\over \pi \nu^{\alpha + {1\over 2}}} \int_0^\infty {dx \over x^{2\alpha}(1+x^2)}\;.
\end{equ}
Similarly, one could modify the linear part of the perturbation. Indeed, given a polynomial $Q$ with $Q(0) = 1$ and positive leading
coefficient, we could consider instead of \eref{e:model} the sequence of problems
\begin{equs}
\d_t u_\eps &= \nu Q\bigl(-\eps^2 \d_x^2\bigr) \d_x^2 u_\eps + f(u_\eps) + \sqrt 2 \xi(t,x)\\
&\qquad + \eps g(u_\eps)\d_x^2 u_\eps + \eps h(u_\eps)\bigl(\d_x u_\eps \otimes \d_x u_\eps\bigr)\;.
\end{equs}
Again, inspection of the proofs given in this article shows that one has convergence to \eref{e:limitmain}, but this time
the prefactor of the correction term is given by
\begin{equ}
{1\over \pi \nu} \int_0^\infty {dx \over Q(x^2)}\;.
\end{equ}
More interestingly, one can also consider situations where the limiting process is deterministic. For example, one can retrace the
steps of the proof of Theorem~\ref{theo:main} for the family of problems given by
\begin{equs}
\d_t u_\eps &= \nu \d_x^2 u_\eps + f(u_\eps) + \sqrt 2 \eps^\gamma \xi(t,x)\\
&\qquad  - \eps^2 \d_x^4 u_\eps + \eps^{1-2\gamma} g(u_\eps)\d_x^2 u_\eps + \eps^{1-2\gamma} h(u_\eps)\bigl(\d_x u_\eps \otimes \d_x u_\eps\bigr)\;,
\end{equs}
and to show that for $\gamma \in (0,{1\over 2})$, the solutions converge to the solutions to the semilinear heat equation
\begin{equ}
\d_t u = \nu \d_x^2 u + \bar f(u)\;,
\end{equ}
with $\bar f(u)$ as in \eref{e:ftilde}.

The borderline case $\gamma = {1\over 2}$ can also be treated, at least when $g = 0$, and we will present this case in detail since it 
appears to be more interesting and arises from a rather natural scaling. Consider the family of problems
$v_\eps$ given by
\begin{equ}[e:mainv]
\d_t v_\eps = \nu \d_x^2 v_\eps - \eps^2 \d_x^4 v_\eps + f(v_\eps) + h(v_\eps)\bigl(\d_x v_\eps \otimes \d_x v_\eps\bigr) + \sqrt {2\eps} \xi(t,x)\;.
\end{equ}
We have the following result, which will be restated rigorously and proven in Section~\ref{sec:thm15} below.

\begin{theorem}\label{theo:convv}
As $\eps \to 0$, the solution $v_\eps$ to \eref{e:mainv} converges to $v$, where
\begin{equ}
\d_t v = \nu \d_x^2 v + \bar f(v) + h(v)\bigl(\d_x v \otimes \d_x v\bigr)\;.
\end{equ}
Here, $\bar f = f + {1\over 2\sqrt \nu} \Tr h$, and the convergence takes place in $H^\beta$ for every $\beta \in ({1\over 2},1)$.
\end{theorem}

Finally, an even more interesting class of generalisations is obtained by removing the prefactor $\eps$ in front of the
nonlinearity altogether, as in \eref{e:mainv}, but keeping a noise term of order $1$ and subtracting the corresponding 
`infinite correction' from the nonlinearity in the hope that a finite object is obtained
in the limit. For example, one would like to be able to study the limit $\eps \to 0$ in
\begin{equ}
\d_t u_\eps = \nu \d_x^2 u_\eps - \eps^2 \d_x^4 u_\eps + h(u_\eps)\bigl(\d_x u_\eps \otimes \d_x u_\eps\bigr) - {\Tr h(u_\eps) \over 2\sqrt \nu \eps} + \sqrt 2 \xi(t,x)\;.
\end{equ}
In the particular case $n=1$ and $h=1$, this reduces to the KPZ equation and it can be shown, using the Hopf-Cole transform,
that a limit exists \cite{MR1462228}, at least for a slightly different kind of approximation. 
Unfortunately, no such tool exists in the general case and the problem of making sense
of a KPZ-type equation directly is still wide open, although some results have been obtained in 
this direction \cite{MR1743612,MR1888875,MR2365646}.

\subsection{Structure of the article}

The remainder of the article is structured as follows. In Section~\ref{sec:motivation}, we show how equations of the type
\eref{e:model} arise in a class of path sampling problems, thus giving an additional motivation to this work. This is followed 
by a brief heuristic calculation which explains why the limiting equation should have the nonlinearity $\bar f$ instead of $f$
and shows how to easily compute the prefactor $1/(2\sqrt \nu)$ or the variations thereof that were discussed in the previous section.
In Section~\ref{sec:apriori}, we give a number of preliminary calculations that are useful for the proof and that show
how the solutions to these equations behave. Section~\ref{sec:averaging} is devoted to the proof of the main technical
bound that allows to control the difference between $\eps(\d_x u_\eps \otimes \d_x u_\eps)$ and a suitable
multiple of the identity matrix in some negative Sobolev norm. This bound is the main ingredient for the
proof of Theorem~\ref{theo:main} which is reformulated in a rigorous way and proved in Section~\ref{sec:main}.

\subsection*{Acknowledgements}

{\small
I would like to thank Sigurd Assing, G\'erard Ben Arous, Dirk Bl\"omker, Weinan E, Jin Feng, Jan Maas, Andrew Majda, Carl Mueller,
Boris Rozovsky, Andrew Stuart, 
Jochen Vo\ss\ and Lorenzo Zambotti for stimulating discussion on this and related problems. I am grateful to the
referee for pointing out the literature on Mosco convergence of Dirichlet forms (see Remark~\ref{rem:Mosco})
and the link to $\Gamma$-convergence.

Financial support was kindly provided by the EPSRC through grants EP/E002269/1 and EP/D071593/1, as
well as by the Royal Society through a Wolfson Research Merit Award.
}

\section{A motivation from path sampling}
\label{sec:motivation}

Besides pure mathematical interest and the, admittedly rather tenuous, connection to the KPZ equation discussed in
the previous section, the study of this
kind of singular limit is motivated by the following path sampling problem.
Consider the ordinary differential equation describing gradient motion of a stochastic particle in a smooth potential $V$:
\begin{equ}[e:gradient]
dq = -\nabla V(q)\,dt + \sqrt{2T}\,dW(t)\;,
\end{equ}
as well as the Langevin equation for a massive particle of mass $m$ subject to the same potential:
\begin{equ}[e:Langevin]
d q = \dot q \,dt\;,\qquad m\, d\dot q = - \nabla V(q)\,dt - \dot q\,dt + \sqrt{2T}\,dW(t)\;.
\end{equ}
(We consider this as a $2n$-dimensional SDE for the pair $(q,\dot q)$.) It is well-known (see for example \cite{MR2099730,MR2130323}
for a recent exposition of this fact in the context of a larger class of problems) that if one takes
the limit $m \to 0$ in the second equation, one recovers the solution to \eref{e:gradient}. This is usually referred to
as the Smoluchowski-Krames approximation to \eref{e:gradient} and can also be guessed
in a `naive' way by simply deleting the term $m\, d\dot q$ appearing in the left hand side of the second identity in
\eref{e:Langevin} and noting that $\dot q\, dt = dq$ from the first identity.

Assume now that $V$ has a critical point at the origin
and consider a `bridge' for both \eref{e:gradient}
and \eref{e:Langevin}. In the case of \eref{e:gradient}, this is a solution starting at the origin and conditioned 
to return to the origin after a fixed time interval (which we assume to be equal to $\pi$ in order to simplify some expressions in the sequel). 
In the case of \eref{e:Langevin},
this is a sample from the stationary solution conditioned to pass through the origin at times $0$ and $\pi$.
Denote by $\mu_m$ the bridge for \eref{e:Langevin} and by $\mu_0$ the bridge for \eref{e:gradient} so that both $\mu_m$ and
$\mu_0$ are probability measures on the space $\CC_0([0,\pi],\R^n)$ of continuous functions vanishing at their endpoints.
  
It was then shown in \cite{HairerStuartVossWiberg05,HairerStuartVoss07} under some regularity and growth assumptions on $V$ that $\mu_0$ is the invariant measure
for the stochastic PDE given by 
\begin{equ}[e:sampleGradient]
du^i = {1\over 2T} \d_t^2 u^i\,d\tau - {1\over 2T} \d_{ij}^2 V(u) \d_jV(u)\, d\tau + {1\over 2} \d_{ijj}^3 V(u)\, d\tau + \sqrt 2 \, dW(\tau)\;,
\end{equ}
where $W$ is a cylindrical process on $L^2([0,\pi],\R^n)$ (so that ${dW \over dt}$ is space-time white noise) 
and the linear operator $\d_t^2$ is endowed with Dirichlet boundary conditions. 
Summation over the index $j$ is implied in those terms where it appears. Note that the variable $t$ plays the role of \textit{space}
in this stochastic PDE, whereas the role of time is played by the `algorithmic time' $\tau$.

Actually, \eref{e:sampleGradient} is even reversible with 
respect to $\mu_0$ and the corresponding Dirichlet form is given by
\begin{equ}[e:Dirichlet]
\CE_0(\phi,\phi) = \int_{L^2([0,\pi],\R^n)} \|\nabla \phi(u)\|^2\,\mu_0(du)\;,
\end{equ}
where $\nabla$ denotes the $L^2$-gradient.

On the other hand, it was shown in \cite{Hypo} that $\mu_m$ is the invariant measure
for the fourth-order SPDE given by
\begin{equs}[e:sampleLangevin]
du^i &= {1\over 2T} \d_t^2 u^i\, d\tau - {m^2 \over 2T} \d_t^4 u^i\,d\tau - {1\over 2T} \d_{ij}^2V(u) \d_jV (u) \,d\tau \\
& \quad - {m\over 2T} \d_{ij\ell}^3 V(u) \d_x u^\ell  \d_x u^j \, d\tau - {m\over T} \d_{ij}^2 V(u) \d_x^2 u^j \, d\tau + \sqrt 2 \, dW(\tau)\;,
\end{equs}
where the linear operator is endowed with Dirichlet boundary conditions (for both $u$ and its second derivative).
Here, summation over repeated indices is also implied. The fact that $0$ is a critical point for $V$ was used in order to derive \eref{e:sampleLangevin}. If this was not the case, then
the boundary conditions for the linear operator would be somewhat more involved, but the final result would be the same.
Again, \eref{e:sampleLangevin} is reversible with respect to $\mu_m$ and the corresponding Dirichlet form $\CE_m$
is given by
\eref{e:Dirichlet} with $\mu_0$ replaced by $\mu_m$.

It is therefore a natural question to ask whether \eref{e:sampleGradient}
can also be recovered as the limit of \eref{e:sampleLangevin} when the mass $m$ tends to zero.  
On the one hand, one expects this to be the case since $\CE_m(\phi,\phi) \to \CE_0(\phi,\phi)$ for sufficiently regular test
functions $\phi$. 

\begin{remark}\label{rem:Mosco}
Note that the convergence of Dirichlet forms when applied to nice test functions does not in general imply the
convergence of the corresponding stochastic processes since it does not in general imply the convergence of the 
resolvents. However, starting with the seminal works of Mosco \cite{MoscoOld,Mosco}, a number of authors have investigated sufficient
conditions for the resolvent convergence of Dirichlet forms, see for example \cite{KuShi,Kol,ASZ,pugachev,Kostja}.
While our setting does not formally seem to be covered by these works, it `morally' falls into
the same category. The main difference is that our result also applies to a class of non-reversible
processes, while we also obtain slightly stronger convergence results.
\end{remark}

On the other hand, taking the `naive' limit 
$m \to 0$ in \eref{e:sampleLangevin} by simply deleting all the terms that contain a factor $m$ or $m^2$
yields
\begin{equ}
du^i = {1\over 2T} \d_x^2 u^i\,d\tau - {1\over 2T} \d_{ij}^2V \d_jV(u)\, d\tau  + \sqrt 2 \, dW(\tau)\;.
\end{equ}
While this is indeed very similar in structure to \eref{e:sampleGradient}, we note that the term containing third-order
derivatives of $V$ is missing! 

Theorem~\ref{theo:main} does indeed explain this apparent discrepancy since
 \eref{e:sampleLangevin} is precisely of the form \eref{e:model} if we make the identifications
\begin{equ}
\eps = {m \over \sqrt{2T}}\;,\quad \nu = {1\over 2T}\;,
\end{equ} 
and define the functions
\begin{equ}
f_i(u) = -{\d_{ij}^2V(u) \d_jV (u) \over 2T}\;,\quad
g_{ij}(u) =  - {2 \d_{ij}^2 V(u)\over \sqrt{2 T}}  \;,\quad h_{ij\ell}(u) = - { \d_{ij\ell}^3 V(u) \over \sqrt{2T}}\;.
\end{equ}
We then have
\begin{equs}
\bar f_i (u) &= f_i(u) + \sqrt{T\over 2} \bigl(h_{ijj}(u) - \d_j g_{ij}(u)\bigr) \\
& = -{1\over 2T}\d_{ij}^2V(u) \d_jV (u) - \sqrt{T\over 2} { \d_{ijj}^3 V(u) -2 \d_{ijj}^3 V(u) \over \sqrt{2T}} \\
& = -{1\over 2T}\d_{ij}^2V(u) \d_jV (u) + {1\over 2} \d_{ijj}^3 V(u)\;.
\end{equs}
It therefore follows at once from Theorem~\ref{theo:main} that \eref{e:sampleGradient} can indeed be recovered
as the limit as $m \to 0$ of \eref{e:sampleLangevin} as expected at the level of the corresponding Dirichlet forms.\footnote{Strictly speaking, Theorem~\ref{theo:main}
does not apply to this situation since there we consider periodic rather than Dirichlet b.c.'s. We do not expect this to make any difference, but some of the arguments 
would be slightly more involved.}

\subsection{Heuristic explanation}

Before we turn to the rigorous formulation of Theorem~\ref{theo:main} and to its proof, let us give a formal argument
why a careful analysis of \eref{e:model} reveals that $\bar f$ should indeed have the form given in \eref{e:ftilde}.  
We expand $u_\eps$ into Fourier modes $u_{\eps,k} \in \R^d$ (we drop the subscript $\eps$ from now on
and simply write $u_k$ instead of $u_{\eps,k}$), so that 
\begin{equ}[e:defek]
u_\eps(x,t) = \sum_{k \in \Z} u_k(t)\,e_k(x)\;,\qquad 
e_k(x) = {\exp(ikx) \over \sqrt{2\pi}} \;.
\end{equ}
Since we only consider real-valued solutions, we furthermore have the constraint $u_{-k} = \bar u_k$.
One would then expect the `high modes' to behave roughly like the linear part of \eref{e:model}, that is for $k \gg 1$
one would expect to have $u_k(t) \approx \psi^\eps_k(t)$ with
\begin{equ}[e:SC]
d\psi_k^{\eps,i} = \bigl(-\nu k^2 - \eps^2 k^4\bigr)\psi_k^{\eps,i}\,dt + \sqrt 2\, dW_k^i(t)\;,
\end{equ} 
where the $W_k^i$'s are i.i.d.\ standard complex-valued Wiener processes. For fixed $t$, this means that $u_k^i(t) \approx \CN(0, (\nu k^2 + \eps^2 k^4)^{-1})$,
which is the invariant measure for \eref{e:SC}. 

We treat the two terms $g(u) \d_x^2 u$ and $h(u) \bigl(\d_x u \otimes \d_x u\bigr)$ appearing in \eref{e:model} separately,
starting with the second term. The main contribution is expected to come from the `diagonal terms' where no cancellations occur.
Therefore, if these terms give a finite contributions to the limiting equation, one can reasonably expect the remaining terms to give
a vanishing contribution. In order to compute the contribution of one of these terms, consider a product of the form
$\phi = v w^2$, where $v$ is a function that has mainly low-frequency components and $w$ has mainly high-frequency components that are
furthermore independent and of zero mean. Since we have the identity
\begin{equ}
e_k(x) e_\ell(x) e_m(x) = {1\over 2\pi} e_{k+\ell+m}(x)\;,
\end{equ}
it follows that the components of $\phi$ are given by
\begin{equ}[e:prodphi]
\phi_n = {1\over 2\pi} \sum_{k+\ell+m = n} v_m w_k w_\ell  \;,
\end{equ}
Since we assumed that the components $w_k$ are independent and that $v_m$ is `small' for `large' values of $m$, it is a reasonable 
expectation that the main contribution
to this sum stems from the terms such that $k = -\ell$ (recall that the $w_k$ are complex-valued random variables so that $\E w_k^2 = 0$, but $\E w_k w_{-k} = \E |w_k|^2 \neq 0$), so that
\begin{equ}
\phi_n \approx {v_n \over 2\pi} \sum_{k \in \Z} |w_k|^2= {v_n \over 2\pi} \|w\|^2\;.
\end{equ}
Turning back to the term $h(u) \bigl(\d_x u \otimes \d_x u\bigr)$, these considerations and the fact that the high-frequency components of $u_\eps$
are expected to be close to those of $\psi^\eps$ suggest that one should have
\begin{equ}
\eps h_{ijk}(u) \bigl(\d_x u_j \d_x u_k\bigr) \approx \eps h_{ijj}(u) \bigl(\d_x \psi^{\eps,j}\bigr)^2 \approx \delta_{kj}{\eps\over 2\pi} h_{ijj}(u)\, \|\d_x \psi^j\|^2\;.
\end{equ}
On the other hand, one has 
\begin{equs}
{\eps \over 2\pi} \|\d_x \psi^i\|^2 &\approx {\eps \over 2\pi} \sum_{k \neq 0} {k^2 \over \nu k^2 + \eps^2 k^4}
\approx {1 \over 2\pi} \sum_{k \in \Z} {\eps \over \nu  + (\eps k)^2} \\
&\approx {1\over 2\pi} \int_{-\infty}^\infty {dk \over \nu +k^2} = {1\over 2\sqrt \nu}\;,
\end{equs}
so that one does indeed expect to have a contribution of
\begin{equ}
\eps h(u) \bigl(\d_x u \otimes \d_x u\bigr) \approx {1\over 2\sqrt \nu} \Tr h(u)\;.
\end{equ}

Let us now turn to the term containing $g$. 
If $g$ was constant, then this term could be absorbed into the constant $\nu$, thus not contributing to
the limit. However, if $g$ is not constant, then the high-frequency components of $g$ can interact with the high-frequency components
of $\d_x^2 u$ in order to produce a non-vanishing contribution to the final result. If the high-frequency part of $u$ is given by $\psi$,
then the high-frequency part of $g(u)$ is in turn expected to be given by $Dg(u) \psi$, so that a reasoning similar to before
yields one to expect
\begin{equ}
\eps g(u) \d_x^2 u \approx \eps Dg(u) \bigl(\psi \otimes \d_x^2\psi\bigr)
\approx \Tr Dg(u) {\eps\over 2\pi} \sum_{|k| \gg 1} {-k^2 \over \nu k^2 + \eps^2 k^4} \approx -{ \Tr Dg(u) \over 2\sqrt \nu}\;.
\end{equ}
One can also reach this conclusion by noting that $\eps g(u) \d_x^2 u = \eps \d_x \bigl(g(u) \d_x u\bigr) - \eps g'(u) \bigl(\d_x u\bigr)^2$
and arguing that the first term should not matter because the first derivative gets `swallowed' by the regularising properties of the
Laplacian. This is actually the argument that will be used later on in the proof.

\section{Preliminary calculations}
\label{sec:apriori}

Let us rewrite the solutions to \eref{e:limitmain} and \eref{e:model} as
\minilab{e:defequ}
\begin{equs}
u(t) &= S(t)u_0 + \int_0^t S(t-s)F(u(s))\,ds + \psi^0(t)\;,\label{e:maineq}\\
u_\eps(t) &= S_\eps(t)u_0 + \int_0^t S_\eps(t-s)F_\eps(u_\eps(s))\,ds + \psi^\eps(t)\;,\label{e:epseq}
\end{equs}
where we set
\minilab{e:defequ}
\begin{equ}[e:defF]
F_\eps(u) = 1+f(u) + \eps g(u) \d_x^2 u + \eps h(u) \bigl(\d_x u \otimes \d_x u\bigr)\;,\quad
F(u) = 1+\bar f(u)\;.
\end{equ}
Here and throughout the remainder of this article, we denote by 
$\psi^\eps$ and $\psi^0$ the \textit{stationary} solutions to the linearised equations
\begin{equs}
d\psi^\eps (t) &= \bigl(\nu \d_x^2 -1\bigr)\psi^\eps(t)\,dt - \eps^2 \d_x^4  \psi^\eps(t)\,dt + \sqrt 2 dW(t)\;,\\
d\psi^0 (t) &= \bigl(\nu \d_x^2 -1\bigr)\psi^0(t)\,dt + \sqrt 2 dW(t)\;.
\end{equs}
Here, $W$ is a standard cylindrical Wiener process on $\CH = L^2([0,2\pi],\R^n)$
and the operators $\CL = \nu \d_x^2-1$ and $\CL_\eps = \nu \d_x^2-1 - \eps^2 \d_x^4$ are both endowed with periodic boundary conditions.

\begin{remark}
The reason why we subtract the constant $1$ to the definitions of $\psi^\eps$ and $\psi^0$ is so that the $0$-mode has also a stationary regime.
This is also why we then add $1$ to the definition of $F$ to compensate.
\end{remark}

We denote by $S(t)$ the semigroup generated by $\CL$ and by $S_\eps(t)$ the semigroup generated by $\CL_\eps$, so that
\begin{equ}
\psi^0 (t) = \sqrt 2 \int_{-\infty}^t S(t-s)\,dW(s)\;,
\end{equ}
and similarly for $\psi^\eps$.

\begin{remark}
Note that $u_0$ is \textit{not} the initial condition for either \eref{e:maineq} or \eref{e:epseq}, since $\psi^0$ and $\psi^\eps$
are \textit{stationary} solutions to the corresponding linear evolution equations. Therefore, one has
$u(0) = u_0 + \psi^0(0)$ and $u_\eps(0) = u_0 + \psi^\eps(0)$.
\end{remark}

We also denote by $H^\alpha$ with $\alpha \in \R$ the usual fractional Sobolev spaces, i.e.\ $H^\alpha$
is the domain of $(-\CL)^{\alpha/2}$, and we denote by $\|\cdot\|_\alpha$ the norm in $H^\alpha$.
It is straightforward to check that for every $\eps > 0$, $\psi^\eps$ takes values in $H^s$ for every $s < {3\over 2}$,
while $\psi^0$ only takes values in $H^s$ for $s < {1\over 2}$.
We first note that one has the following local existence and uniqueness results \cite{DaPrato-Zabczyk92,Hairer08}:

\begin{lemma}
Equation~\ref{e:maineq} admits a unique local mild solution for every $u_0 \in L^\infty$
and \eref{e:epseq} admits a unique local mild solution for every $u_0 \in H^1$.
\end{lemma}

\begin{proof}
For \ref{e:maineq}, it suffices to note that $\psi^0$ has continuous sample paths with values in $L^\infty$ \cite{DaPrato-Zabczyk92}
and that $F$ is locally Lipschitz from $L^\infty$ into itself, so that a standard Picard fixed point argument applies.
Similarly, $\psi^\eps$ has continuous sample paths with values in $H^1$ and it is easy to check that 
$F_\eps$ is locally Lipschitz from $H^1$ into $H^{-1}$. Since furthermore $\|S_\eps(t)\|_{H^{-1} \to H^1} \le C_\eps / \sqrt t$,
the claim follows again from a standard fixed point argument.
\end{proof}

\begin{remark}
Note that at this stage we do not make any claim on the uniformity of bounds on the solutions to \eref{e:epseq} as $\eps \to 0$.
\end{remark}

In particular, this implies that, for every initial condition $u_0 \in L^\infty$, there exists a stopping time
$\tau^*$ such that \eref{e:maineq} has a unique mild solution for $t < \tau^*$ and such that
either $\tau^* = \infty$ or $\lim_{t \to\tau^*} \|u(t)\|_{L^\infty} = \infty$. 

\subsection{Semigroup}

In this subsection, we give a few convenient bounds on the semigroups $S$ and $S_\eps$. Recall that 
for every $\alpha \ge \beta$, every $T>0$ and every $\gamma > 0$, there exists a constant $C$ such that 
\begin{equ}[e:sg]
\bigl\|\exp\bigl(- (-\CL)^\gamma t\bigr)u\bigr\|_\alpha \le C t^{\beta - \alpha \over 2\gamma} \|u\|_\beta\;,
\end{equ}
for every $u \in H^\beta$ and every $t \le T$. This follows from standard semigroup theory and can also 
be checked directly by decomposing into Fourier modes.

Using this, we start by giving a sharp bound on the regularising properties of $S_\eps$ as a function of the parameter $\eps > 0$.
We have:
\begin{lemma}\label{lem:smoothSeps}
For every $\alpha \ge \beta$ and $T>0$, there exists $C>0$ such that the bound
\begin{equ}
\|S_\eps(t)u\|_\alpha \le C \bigl(t^{\beta-\alpha \over 2} \wedge (\eps^2 t)^{\beta - \alpha \over 4}\bigr) \|u\|_\beta\;,
\end{equ}
holds for every $\eps \le 1$ and every $t \le T$.
\end{lemma}

\begin{proof}
Since $\CL_\eps \le \CL - K_1 \eps^2 \CL^2 + K_2$ for some constants $K_i$, we have
\begin{equ}
\|S_\eps(t)u\|_\alpha \le C \bigl(\|S(t)u\|_\alpha \wedge \|\exp(-K_1 \eps^2 \CL^2 t)u\|_\alpha\bigr)\;,
\end{equ}
so that the bound follows immediately from \eref{e:sg}.
\end{proof}

\begin{lemma}\label{lem:diffSu}
For every $\gamma < 2$, every $\beta \le \alpha + 2\gamma$, and every $T>0$, there exists a constant $C$ such that the bound
\begin{equ}[e:boundSdiff]
\|S(t)u - S_\eps(t)u\|_\alpha \le C \eps^{\gamma} t^{\beta - \alpha - \gamma \over 2} \|u\|_{\beta}
\end{equ}
holds for every $t \le T$ and every $u \in H^\beta$.
\end{lemma}

\begin{proof}
Denote by $u_k$ the $k$th Fourier mode of $u$, by $r_k$ the $k$th Fourier mode of $r \eqdef S(t)u - S_\eps(t)u$, and set
$s_k(t) = e^{-(\nu k^2 + 1)t}$.
We then have
\begin{equ}
|r_k| =  s_k(t) \bigl|\bigl(e^{- \eps^2 k^4 t} - 1\bigr)u_k \bigr| \le s_k(t) \bigl(k^4 \eps^2 t \wedge 1\bigr)|u_k|
\le \eps^{\gamma} k^{2\gamma} t^{\gamma \over 2} s_k(t) |u_k|\;.
\end{equ}
It follows that $\|r\|_\alpha \le C \eps^\gamma t^{\gamma \over 2}\|S(t)u\|_{\alpha + 2\gamma}$, so that the result follows again from \eref{e:sg}.
\end{proof}

\begin{remark}
Note that as operators from $H^\alpha$ to $H^\alpha$, one has for every $T>0$
\begin{equ}
\liminf_{\eps \to 0} \sup_{t \in [0,T]} \|S(t) - S_\eps(t)\|_\alpha > 0\;,
\end{equ}
as can easily be seen by looking at the operators acting on $e_k$ for suitably chosen $k$. This explains the presence of a divergence
as $t \to 0$ in \eref{e:boundSdiff} if one takes $\alpha = \beta$ and $\gamma > 0$. 
\end{remark}

\subsection{Stochastic convolution}

We first recall the following quantitative version of Kolmogorov's continuity test:

\begin{proposition}\label{prop:Kolmo}
Let $\Psi$ be a Gaussian random field on $[0,1]^d$ with values in a separable Hilbert space $\CH$
such that there exist constants $K$, and $\alpha > 0$ such that 
\begin{equ}
\E \| \Psi(x) - \Psi(y)\|^2 \le K |x-y|^\alpha \;,
\end{equ}
holds for every $x$, $y$. Then, for every $p>0$ there exist a constant $C$ such that 
\begin{equ}
\E \sup_{x \in [0,1]^d} \|\Psi(x)\|^p \le C \bigl|K + \E \|\Psi(0)\|^2\bigr|^{p\over 2}\;.
\end{equ}
\end{proposition}

\begin{proof}
It suffices to keep track of constants in the standard proof of Kolmogorov's criterion as in \cite{MR1083357}.
It also follows immediately from the Fernique-Talagrand result on the supremum of Gaussian processes \cite{MR0413237,MR906527}.
\end{proof}

Denoting by $\psi^\eps_k(t)$ the scalar product $\scal{\psi^\eps(t),e_k}$, we then have
\begin{equ}
\E \bigl|\psi^\eps_k\bigr|^2(t) = {1\over 1+\nu k^2 + \eps^2 k^4}\;.
\end{equ}
It follows immediately that one has the following \textit{a priori} bounds on $\psi^\eps$:

\begin{lemma}\label{lem:boundpsi}
There exist constants $C>0$ and $C_\alpha$ such that the bounds
\begin{equs}[2]
\E \|\psi^\eps(t)\|_\alpha^2 &\le C_\alpha \;,& \alpha & \in (-\infty, \hf)\;,\\
\E \|\psi^\eps(t)\|_{\alpha}^2 &\le C |\log \eps| \;,& \alpha&= \hf\;,\\
\E \|\psi^\eps(t)\|_\alpha^2 &\le C_\alpha \eps^{1-2\alpha} \;,& \quad \alpha& \in (\hf, \textstyle{3\over 2})\;,
\end{equs}
hold for every $\eps \in (0,1]$. Furthermore, for every $\delta > 0$ and every $T>0$, there exist constants $\tilde C_\alpha$ such that we have the bounds
\begin{equs}[2]
\E \sup_{t \in [0,T]}\|\psi^\eps(t)\|_\alpha^2 &\le \tilde C_\alpha \;,& \alpha & \in (-\infty, \hf)\;,\\
\E \sup_{t \in [0,T]}\|\psi^\eps(t)\|_\alpha^2 &\le \tilde C_\alpha \eps^{1-2\alpha - 2\delta} \;,& \quad \alpha& \in [\hf, \textstyle{3\over 2})\;.
\end{equs}
\end{lemma}

\begin{proof}
The first set of bounds follows from a straightforward explicit calculation. The second set of bounds follows similarly by 
applying Proposition~\ref{prop:Kolmo}. For example, for $\alpha \in [\hf, {3\over 2})$ and $\gamma \in (0,{3\over 4}-{\alpha \over 2})$, one obtains the bound
\begin{equs}
\E \|\psi^\eps(t) - \psi^\eps(s)\|_\alpha^2 &\le \sum_{k\in \Z} {k^{2\alpha} \over 1+\nu k^2 + \eps^2 k^4} \bigl(1 \wedge |t-s| (\nu k^2 + \eps^2 k^4)\bigr) \\
&\le |t-s|^\gamma \sum_{k\in \Z} {k^{2\alpha} \over \bigl(1+\nu k^2 + \eps^2 k^4\bigr)^{1-\gamma}} \\
&\approx \eps^{1-2\alpha - 2\gamma} |t-s|^\gamma \int_{-\infty}^\infty {x^{2(\alpha + \gamma - 1)} \over (\nu + x^2)^{1-\gamma}}\,dx\;,
\end{equs}
which implies the required bound by Proposition~\ref{prop:Kolmo}. The other bounds follow similarly.
\end{proof}

Note that the function $\alpha \mapsto C_\alpha$ is bounded on any closed interval not including $\hf$ or ${3\over 2}$, but it 
diverges at these two values. Similarly, the function $\tilde C_\alpha$ diverges as it approaches $\hf$ from below.
We now give a bound on the speed at which $\psi^\eps$ approaches $\psi^0$ as $\eps \to 0$.

\begin{proposition}\label{prop:boundGauss}
For every $\kappa > 0$ and every $T>0$, there exists a constant $C$ such that
\begin{equ}
\E \sup_{t\in [0,T]}\|\psi^\eps(t) - \psi^0(t)\|_{L^\infty} \le C \eps^{{1\over 2}- \kappa}\;,
\end{equ}
for every $\eps < 1$.
\end{proposition}

\begin{proof}
Define the random fields
\begin{equ}
\Delta_\eps^k(x,t) = \bigl(\psi_k^\eps(t) - \psi_k^0(t)\bigr)\,e_k(x)\;,
\end{equ}
so that
\begin{equ}
\Delta_\eps \eqdef \psi^\eps - \psi^0 = \sum_{k \in \Z} \Delta_\eps^k\;,
\end{equ}
and the fields $\Delta_\eps^k$ are all independent. Let us first obtain a bound on the second moment
of $\Delta_\eps^k(t,x)$.
For $k=0$, we have $\Delta_\eps^0 = 0$, while for  $k \neq 0$ we have the bound
\begin{equ}
\E |\psi^\eps_k(0) - \psi^0_k(0)|^2 \le \int_0^\infty e^{- 2 \nu k^2 t} \bigl(1 - e^{-\eps^2 k^4 t}\bigr)^2\,dt\;. 
\end{equ}
Noting that there exists a constant $C$ such that
\begin{equ}
 \bigl(1 - e^{-\eps^2 k^4 t}\bigr)^2 \le C \bigl(1 \wedge \eps^4 k^8 t^2\bigr)\;,
\end{equ}
we can break this integral into two parts. For the first part, we obtain
\begin{equ}
\eps^4 k^8\int_0^{\eps^{-2} k^{-4}} e^{- 2 \nu k^2 t}t^2\,dt = \eps^4 k^2 \int_0^{\eps^{-2}k^{-2}} e^{-2\nu t} t^2\,dt
\le C \bigl(\eps^4 k^2 \wedge \eps^{-2} k^{-4}\bigr)\;.
\end{equ}
For the second part, we obtain
\begin{equs}
\int_{\eps^{-2} k^{-4}}^\infty e^{- 2 \nu k^2 t}\,dt &= k^{-2}\int_{\eps^{-2} k^{-2}}^\infty e^{- 2 \nu t}\,dt
\le Ck^{-2}  \Bigl(1 \wedge \int_{\eps^{-2} k^{-2}}^\infty {dt \over t^3}\Bigr) \\
&= C \bigl(k^{-2} \wedge \eps^4 k^2\bigr)\;.
\end{equs}
Note now that if $k^{-2} < \eps^4 k^2$, then one has $\eps^{-2} k^{-4} < k^{-2}$ so that, combining 
these two bounds, we conclude that
\begin{equ}[e:bounddiffpsi]
\E |\psi^\eps_k(0) - \psi^0_k(0)|^2 \le C \bigl(k^{-2} \wedge \eps^4 k^2\bigr)\;.
\end{equ}
In particular $\E \bigl(\Delta_\eps^k(t,x)\bigr)^2 \le C \bigl(k^{-2} \wedge \eps^4 k^2\bigr)$ by stationarity and
the fact that the functions $e_k$ are uniformly bounded.

On the other hand, $\psi_k^0(t)$ is a stationary Ornstein-Uhlenbeck process with both
characteristic time and variance $1/(1+\nu k^2)$. We thus
 have 
\begin{equs}
\E \bigl(e_k(x)\psi_k^0(t) - e_k(y)\psi_k^0(s)\bigr)^2 &\le C\E |\psi_k^0(t) - \psi_k^0(s)|^2 + Ck^2 |x-y|^2 \E |\psi_k^0(s)|^2 \\
&\le C|t-s| + C |x-y|^2\;,
\end{equs}
and it can easily be checked that the same bound also holds for $\psi_k^\eps(t)$, uniformly in $\eps$.
It follows from this and \eref{e:bounddiffpsi} that
\begin{equs}
\E \bigl(\Delta_\eps^k(x,t) - \Delta_\eps^k(y,s)\bigr)^2 &\le C \bigl(k^{-2} \wedge \eps^4 k^2 \wedge (|t-s| + |x-y|^2)\bigr)\\
&\le C\eps^{1-2\kappa} k^{-1 - \kappa} (|t-s| + |x-y|^2)^{\kappa \over 4}\;.
\end{equs}
Since the $\Delta_\eps^k$ are independent for different values of $k$, we then obtain
\begin{equs}
\E \bigl(\Delta_\eps(x,t) - \Delta_\eps(y,s)\bigr)^2 &\le \sum_{k \ge 1}\E \bigl(\Delta_\eps^k(x,t) - \Delta_\eps^k(y,s)\bigr)^2 \\
&\le  C\eps^{1-2\kappa} (|t-s| + |x-y|^2)^{\kappa \over 4}\;,
\end{equs}
so that the claim follows from Proposition~\ref{prop:Kolmo}.
\end{proof}

\section{Averaging results}
\label{sec:averaging}

The aim of this section is to show that, for any sufficiently regular function $v$, one has
\begin{equ}
v \cdot \bigl(\eps \d_x \psi^\eps(t)\otimes \d_x \psi^\eps(t) - A\bigr) \to 0\;,
\end{equ}
in a weak sense as $\eps \to 0$, where $A$ is a shorthand for ${1/ (2\sqrt \nu)}$ times the identity matrix. 
In fact, if $v \in H^\alpha$ for some $\alpha > {1\over 2}$ and the convergence is measured in
a weak Sobolev space of sufficiently negative index (less than ${-{1\over 2}}$), then the convergence takes place at speed
$\CO(\sqrt \eps)$.

Let $\{w_k^\eps\}_{k \in \Z}$ be a sequence of centred complex Gaussian random variables that are independent, except for the constraint
$w_{-k} = \bar w_k$, and depending on a parameter
$\eps > 0$ with variance
$\sigma_k = {k^2 \over 1+\nu k^2 + \eps^2 k^4}$. We also set $\hat \sigma_k = {1 \over \nu + \eps^2 k^2}$, and we note that
one has
\begin{equ}[e:diffsigma]
|\sigma_k - \hat \sigma_k| \le {C\over 1+k^2}\;.
\end{equ}
One should think of the $w_k$ as being the Fourier coefficients of one of the components of $\d_x \psi^\eps(t)$
for any fixed time $t$. We also consider $\{\tilde w_k^\eps\}_{k \in \Z}$ a sequence with the same distribution but independent
of the $w_k$'s. Think of $\tilde w$ as describing a different component from the one that $w$ describes.

Let furthermore $v_k$ be a (possibly $\eps$-dependent) sequence of complex random variables such that $v \in H^\alpha$ almost surely
for some $\alpha > {1\over 2}$. With this notation, the $n$th Fourier modes of $\phi = v(w^2- {1/ (2\eps\sqrt \nu)})$ and $\tilde \phi = vw\tilde w$ are given by
\begin{equs}
\phi_n &= {1\over 2\pi}\sum_{k+\ell + m = n } v_m w_k w_\ell - {v_n \over 2\eps\sqrt \nu} \;,\\
\tilde \phi_n &= {1\over 2 \pi} \sum_{k+\ell + m = n} v_m w_k \tilde w_\ell\;.
\end{equs}
The aim of this section is to give sharp bounds on these quantities.
Our main result in this setting is the following

\begin{proposition}\label{prop:varbit}
Let $v$ be a random variable such that $\|v\|_\alpha < \infty$ almost surely for some fixed $\alpha > {1\over 2}$. Then, for every $\gamma > {1\over 2}$, there exist
 constants $C$, $c$ and $\delta$ such that the bound
\begin{equ}[e:boundphi]
\P (\|\phi\|_{-\gamma} \ge K) \le \sum_{n \ge 1} \P \bigl(\|v\|_\alpha \ge Vn^\delta\bigr) + C \exp\Bigl(-c{\sqrt \eps K \over V}\Bigr)\;,
\end{equ}
hold for every $K>0$ and every $V>0$, and similarly for $\tilde \phi$.
\end{proposition}

The main ingredient in the proof of Proposition~\ref{prop:varbit} is the Gaussian
concentration inequality for Lipschitz continuous functions \cite{SudTsi,Borell,Talagrand}, which 
we state here in the following form:

\begin{proposition}\label{prop:GaussConc}
Let $G\colon \R^N \to \R$ be Lipschitz continuous with Lipschitz constant $L$
and let $X$ be a normal $\R^N$-valued Gaussian random variable. Then, there exists $c> 0$
 independent of $N$ such that the bound
\begin{equ}[e:boundconc]
\P \bigl(\bigl|G(X) - \E G(X)\bigr| \ge K\bigr) \le \exp\Bigl(- c {K^2 \over L^2}\Bigr)\;,
\end{equ}
holds for every $K>0$.
\end{proposition}

\begin{remark}\label{rem:conc}
The bound \eref{e:boundconc} also holds if $X$ is centred Gaussian with a covariance matrix different from the identity,
provided that the eigenvalues of the covariance matrix are bounded by some constant independent of $N$. In this case,
the optimal constant $c$ depends on that bound.
\end{remark}

Another useful tool is the following little calculation: 

\begin{lemma}\label{lem:sumexp}
For every $\delta > 0$ and every $\alpha < 1$, there exists a constant $C > 0$ such that the bound
\begin{equ}
1\wedge \sum_{n \ge 1} \exp(-K n^\delta) \le C \exp(-\alpha K)\;,
\end{equ}
holds for every $K>0$.
\end{lemma}

\begin{proof}
The bound is trivial for $K<1$, so that we assume $K \ge 1$ in the sequel. Since for every $\alpha < 1$ there exists 
a constant $c$ such that $(1+x)^\delta \ge \alpha + c x^\delta$, we have
\begin{equs}
\sum_{n \ge 1} \exp(-K n^\delta) &\le e^{-K} + \int_1^\infty  \exp(-K x^\delta)\, dx
= e^{-K} + \int_0^\infty  \exp(-K (1+x)^\delta)\, dx \\
&\le e^{-K} + e^{-\alpha K}\int_0^\infty  \exp(-K c x^\delta)\, dx = e^{-K} + C K^{-{1\over \delta}} e^{-\alpha K}\;,
\end{equs}
as claimed.
\end{proof}

We have now all the tools required for the

\begin{proof}[of Proposition~\ref{prop:varbit}]
We start with the bound on $\phi$.
Note first that since $\gamma > {1\over 2}$, there exists a constant $C$ and some $\kappa > 0$ (any value less than $\gamma - {1\over 2}$ will do) such that the implication
\begin{equ}
\|\phi\|_{-\gamma} \ge K \qquad \Rightarrow\qquad \exists n\;:\; |\phi_n| \ge CKn^\kappa
\end{equ}
holds for every $K>0$. Furthermore, we see from the definition of $\phi$ that there exist elements $f^{(n)}$ such that
$\phi_n$ can be written as $\phi_n = \scal{v, f^{(n)}}$. In particular, $|\phi_n| \le \|v\|_\alpha \|f^{(n)}\|_{-\alpha}$, so that one has the implication
\begin{equ}
 |\phi_n| \ge CKn^\kappa \qquad \Rightarrow\qquad \|v\|_\alpha \ge V n^\delta\quad\text{or}\quad
 \|f^{(n)}\|_{-\alpha} \ge C{Kn^{\kappa -\delta}\over V}\;.
\end{equ}
Combining these implications, we obtain the bound
\begin{equ}[e:boundphifirst]
\P (\|\phi\|_{-\gamma} \ge K) \le \sum_{n \ge 1} \P \bigl(\|v\|_\alpha \ge Vn^\delta\bigr) 
+ \sum_{n \ge 1} \P \Bigl(\|f^{(n)}\|_{-\alpha} \ge C{Kn^{\kappa -\delta}\over V}\Bigr) \;.
\end{equ}
Since the first term on the right hand side is exactly of the desired form, it remains to obtain suitable bounds on the $f^{(n)}$.
Note at this point that if we want to have any chance to obtain the desired bound, we need to choose $\delta$ in such a way that 
$\kappa - \delta > 0$, which imposes the restriction $\delta < \gamma - {1\over 2}$.

We now break $f^{(n)}$ into a `diagonal' part and a `off-diagonal' part by writing $f^{(n)} = f^{(n,1)} + f^{(n,2)}$ with
\begin{equs}
f^{(n,1)}_m &= \Bigl({1\over 2\pi}\sum_{k \in \Z} |w_k|^2 - {1 \over 2\eps\sqrt \nu} \Bigr)\delta_{-n,m}\;, \\
f^{(n,2)}_m &= {1\over 2\pi} \sump w_{k} w_{\ell}\;.
\end{equs}
Here, the ${}^\prime$ over the second sum indicates that we omit the terms with $k+\ell=0$. 
The first term can be bounded in a straightforward manner. We have
\begin{equ}
 \|f^{(n,1)}\|_{-\alpha} \le {1 \over 2\pi |n|^{\alpha}}\; \Bigl|{T_\nu \over \eps} - \sum_{k \in \Z}|w_k|^2\Bigr|\;,
\end{equ}
where we set $T_\nu =  {\pi \over \sqrt \nu}$ as a shorthand.
At this stage, we note that $\sum_{k} \eps \hat\sigma_k$ is nothing but a Riemann sum for the integral $\int_{-\infty}^\infty {dx\over \nu+x^2}$.
Since the value of this integral is precisely equal to $T_\nu$ and since the function $1/(\nu + x^2)$ is decreasing on $[0,\infty)$ and since $\sum_k |\sigma_k - \hat \sigma_k| \le C$ by \eref{e:diffsigma}, we find
that $|T_\nu - \sum_{k} \eps \sigma_k| \le C \eps$ for some constant $C$. 
Applying Proposition~\ref{prop:GaussConc} to the function $G(w) = \sqrt{\sum_{k \in \Z}|w_k|^2}$, it follows that the bound
\begin{equ}[e:boundphi1]
\P \bigl( \|f^{(n,1)}\|_{-\alpha} \ge K\bigr) \le C \exp\bigl(- c \sqrt \eps {K n^\alpha}\bigr)\;,
\end{equ}
holds for some constants $c$ and $C$, uniformly in $K$.

We now aim for a similar bound for $f^{(n,2)}$. In order to do this, we make use of the following trick.
Setting 
\begin{equ}
f_n = \Bigl(\eps\Bigl(\sum_{k \in \Z} |w_k|^2\Bigr)^2 +  \sum_{m \in \Z} (1+|m|)^{-2\alpha} \Bigl|\sump w_{k} w_{\ell}\Bigr|^2\Bigr)^{1/4}\;,
\end{equ}
we have the almost sure bound $\|f^{(n,2)}\|_{-\alpha} \le 2 f_n^2$, so that it is sufficient to be able to get bounds on the $f_n$.
Of course, this would also be true if the definition of $f_n$ didn't include the term proportional to $\eps$, but we will see shortly that 
this term is very useful in order to be able to apply Proposition~\ref{prop:GaussConc}.

We first obtain a bound on the expectation of $f_n^4$:
\begin{equ}
\E |f_n|^4 \le \eps\, C\Bigl(\E \sum_{k\in \Z}|w_k|^2\Bigr)^2 + \sum_{m \in \Z}(1+|m|)^{-2\alpha}\E \Bigl|\sump w_{k} w_{\ell}\Bigr|^2\;.
\end{equ}
Since $\sum_{k}\sigma_k = \CO(\eps^{-1})$, the first term in this sum is of order $\eps^{-1}$.
Furthermore, since the $w_k$'s are independent, only the `diagonal' terms remain when expanding the square 
under the second expectation. We thus have
\begin{equs}
\E |f_n|^4 &\le {C\over \eps}+C\sum_{m \in \Z} (1+|m|)^{-2\alpha} \Bigl(\sump \sigma_k \sigma_\ell\Bigr)  \\
&\le  {C\over \eps}+C\sum_{m \in \Z} (1+|m|)^{-2\alpha} \; \sum_{k \in \Z} \sigma_k \le {C\over \eps}\;, \label{e:meanfn}
\end{equs}
for some universal constant $C$, provided that $\alpha > {1\over 2}$. Here, we have used the fact that $\sum_{k}\sigma_k = \CO(1/\eps)$,
as already mentioned previously.
In particular, the bound \eref{e:meanfn} implies that $\E |f_n| \le C \eps^{-1/4}$ by Jensen's inequality.

Our next step is to obtain a bound on the Lipschitz constant of $f_n$ as a function of the $w$'s. Denoting by $\d_j$ 
the partial derivative with respect to $w_j$, we have 
\begin{equ}
2f_n^3 \d_j f_n = \eps w_{-j} \sum_{k\in\Z}|w_k|^2 + \sum_{m \in\Z}(1+|m|)^{-2\alpha} w_{m+n-j}\sump w_{k} w_{\ell}\;.
\end{equ}
It then follows from the Cauchy-Schwarz inequality that
\begin{equ}
2|f_n^3 \d_j f_n| \le |f_n|^2 \Bigl(\sqrt \eps |w_{-j}|  + \sqrt{ \sum_{m \in \Z}(1+|m|)^{-2\alpha} |w_{m+n-j}|^2}\Bigr)\;.
\end{equ}
Setting $\|Df_n\|^2 = \sum_{j \ge 1} |\d_j f_n|^2$ for the norm of the derivative of $f_n$, it follows that
\begin{equ}[e:boundDf]
\|Df_n\|^2 \le C|f_n|^{-2} \Bigl(\eps \sum_{\ell \in\Z} |w_\ell|^2 + \sum_{k \in\Z} |w_k|^2 \sum_{m\in \Z} (1+|m|)^{-2\alpha}\Bigr)\;.
\end{equ}
Since $\alpha > {1\over 2}$ by assumption, the last sum in this term is finite. We see now why it was useful to 
add this additional term proportional to $\eps$ in the definition of $f_n$. This term indeed now allows us to obtain the bound
\begin{equ}
\sum_{k \in\Z} |w_k|^2 \le {1\over \sqrt \eps} |f_n|^2\;,
\end{equ}
(which would not hold otherwise!), thus yielding $\|Df_n\|\le C \eps^{-1/4}$. Combining this with \eref{e:meanfn},
we conclude that the bound
\begin{equs}
\P \bigl(|f_n| \ge K\bigr)  \le \P \Bigl(|f_n - \E f_n| \ge K - C \eps^{-1/4}\Bigr) \le C\exp \Bigl(-c \sqrt\eps K^2\Bigr)\;,
\end{equs}
holds by Remark~\ref{rem:conc}, so that
\begin{equ}
\P \bigl(\|f^{(n,2)}\|_{-\alpha} \ge K\bigr)  \le C\exp \bigl(-c \sqrt\eps K\bigr)\;.
\end{equ}
Combining this bound with \eref{e:boundphi1}, we conclude that
\begin{equ}[e:boundphin]
\P \bigl(\|f^{(n)}\|_{-\alpha} \ge K\bigr) \le C\exp \bigl(-c \sqrt\eps K\bigr)\;,
\end{equ}
for possibly different constants $c$ and $C$. Inserting this into \eref{e:boundphifirst}, we obtain
\begin{equ}
\P (\|\phi\|_{-\gamma} \ge K) \le \sum_{n \ge 1} \P \bigl(\|v\|_\alpha \ge Vn^\delta\bigr) 
+ \sum_{n \ge 1}C\exp \Bigl(-c \sqrt\eps {Kn^{\kappa -\delta}\over V}\Bigr)\;,
\end{equ}
so that the desired bound follows from Lemma~\ref{lem:sumexp}.

We now turn to the bound for $\tilde \phi$, which is obtained in a very similar way. Similarly to before, we have $\tilde \phi_n = \scal{v,g^{(n)}}$
for elements $g^{(n)}$ given by
\begin{equ}
g^{(n)}_m = \sum_{k+\ell = m+n}w_k \tilde w_\ell\;.
\end{equ}
Similarly to before, we set
\begin{equ}
g_n^4 = \eps \Bigl(\sum_{m \in\Z}|w_m|^2\Bigr)^2 + \eps \Bigl(\sum_{m \in\Z}|\tilde w_m|^2\Bigr)^2 + \sum_{m \in\Z} (1+|m|)^{-2\alpha} \Bigl(\sum_{k+\ell = m+n}w_k \tilde w_\ell\Bigr)^2\;,
\end{equ}
so that the bound $\|g^{(n)}\|_{-\alpha} \le |g_n|^2$ holds almost surely.

Let us first compute the expectation of $|g_n|^4$.
Expanding the last square, we see that as previously, since $w$ and $\tilde w$ are independent, only the `diagonal' terms contribute, so that
\begin{equs}
\E |g_n|^4 &\le {C\over \eps} + \sum_{m \in\Z} (1+|m|)^{-2\alpha} \sum_{k+\ell = m+n}\sigma_k \sigma_\ell \le {C \over \eps} \;.
\end{equs}
Here, we have used the fact that $\sigma_k \le {1/\nu}$ to go from the first to the second line.

Since $g_n$ is symmetric in $w$ and $\tilde w$, it suffices to obtain a bound 
of the form $\|Dg_n\| \le C\eps^{-1/4}$ to conclude in the same way as before. We have the identity
\begin{equ}
2g_n^3 \d_j g_n = \eps w_{-j} \sum_{k\in\Z}|w_k|^2 + \sum_{m \in \Z}(1+|m|)^{-2\alpha} \tilde w_{m+n-j}\Bigl(\sum_{k+\ell = m+n}w_k \tilde w_{\ell}\Bigr)\;,
\end{equ}
from which a bound similar to \eref{e:boundDf} (only with $w$ replace by $\tilde w$ in the second term) follows. 
The rest of the argument is identical to before.
\end{proof}

We will mostly make use of the following two special cases. The first one is obvious from \eref{e:boundphi}:

\begin{corollary}\label{cor:conc}
Let  $\alpha, \gamma > {1\over 2}$ and let $v$ be a random variable such that $\|v\|_\alpha \le V$ almost surely. Then,
\begin{equ}
\P (\|\phi\|_{-\gamma} \ge K) \le C \exp\Bigl(-c{\sqrt \eps K \over V}\Bigr)\;,
\end{equ}
and similarly for $\tilde \phi$.
\end{corollary}

The second corollary does require a short calculation:

\begin{corollary}\label{cor:boundGauss}
In the setting of Proposition~\ref{prop:varbit}, let $\tilde v_n$ be a collection of independent centred Gaussian random variables
such that $\E \tilde v_n^2 = {\sigma_n \over n^2}$ and assume that $v$ is a random variable such that $\|v\|_\alpha \le 1+\|\tilde v\|_\alpha$
almost surely. Then,
\begin{equ}[e:boundvGauss]
\P (\|\phi\|_{-\gamma} \ge K) \le C \exp\Bigl(-c \eps^{2\alpha/3} K^{2/3}\Bigr)\;,
\end{equ}
and similarly for $\tilde \phi$.
\end{corollary}

\begin{proof}
A straightforward calculation shows that 
\begin{equ}
\E \|\tilde v\|_\alpha^2 \le \sum_{n \ge 1} {1 \over n^{2-2\alpha}(\nu + \eps^2 n^2)} = \CO (\eps^{1-2\alpha})\;,
\end{equ}
so that by Fernique's theorem there exists a constant $c$ such that 
\begin{equ}
\P \bigl(\|\tilde v\|_{\alpha} > K\bigr) \le \exp \bigl(- c \eps^{2\alpha-1} K^2\bigr)\;,
\end{equ}
and similarly for $v$.
It therefore follows from \eref{e:boundphi} and Lemma~\ref{lem:sumexp} that
\begin{equ}
\P (\|\phi\|_{-\gamma} \ge K) \le  C\exp \bigl(- c \eps^{2\alpha-1} V^2\bigr) + C\exp \Bigl(- c \sqrt \eps {K\over V}\Bigr)\;.
\end{equ}
Setting $V = \eps^{{1\over 2} - {2\alpha \over 3}} K^{1\over 3}$ completes the proof.
\end{proof}

\begin{remark}
Here we did not make any assumption regarding the correlations between the $v_n$ and the $w_n$. In the special cases
when either $v_n = {w_n \over n}$ or $v$ is independent of $w$ and $\tilde w$, it is possible to check by an explicit calculation
that the boundary $\alpha = {1\over 2}$ can be reached in \eref{e:boundvGauss}. Since this is of no particular use to us, there
is no need to make this additional effort.
\end{remark}

\section{Proof of the main result}
\label{sec:main}

This section is devoted to the proofs of Theorems~\ref{theo:main} and \ref{theo:convv}.
We first reformulate Theorem~\ref{theo:main} in a more precise way:
\begin{theorem}\label{theo:mainrigor}
Let $u_0$ be an $H^{3\over 2}$-valued random variable. Then, for every $\kappa > 0$ and every $T>0$, there exists a sequence of stopping times $\tau_\eps$ with $\tau_\eps \to T \wedge \tau^*$ in probability,
so that
\begin{equ}
\lim_{\eps \to 0} \P \Bigl(\sup_{t\in [0,\tau_\eps]} \|u(t) - u_\eps(t)\|_{L^\infty} > \eps^{{1\over 2}-\kappa}\Bigr) = 0\;.
\end{equ}
\end{theorem}

\begin{remark}
Assuming that $u_0$ takes values in $H^{3 \over 2}$ may look unnatural at first sight since $u(t)$
do not take values in that space. Note however that $u_0$ is \textit{not} the initial condition to \eref{e:limitmain} since
we have assumed that $\psi^0$ is a stationary realisation of the solution to the linearised problem. Therefore, the 
initial condition for the Cauchy problem \eref{e:limitmain} is given by $u_0 + \psi^0(0)$. It turns out (this follows
from standard analytic semigroup techniques) that the
solution to \eref{e:limitmain} with any continuous initial condition is of the form 
\begin{equ}
u(t) = v(t) + \psi^0(t)\;,\quad v(t) \in H^{3\over 2}\;,
\end{equ}
for any $t > 0$. (Actually, it is even true that $v(t) \in H^s$ for every $s < {5\over 2}$, but this is irrelevant for
our study.)
\end{remark}

\begin{remark}
If the limiting problem has global solutions, then one has $\tau_\eps \to T$ in probability. However, this does \textit{not} imply that the
approximating problem has global solutions! It only implies that its explosion time becomes arbitrarily large as $\eps \to 0$.
This is also why it would be unrealistic to expect approximation results in expectation rather than in probability.
\end{remark}

In order to prove Theorem~\ref{theo:mainrigor}, we consider the following two intermediary problems:
\begin{equs}
u^{(1)}_\eps(t) &= S(t)u_0 + \int_0^t S(t-s)F(u^{(1)}_\eps(s))\,ds + \psi^\eps(t)\;,\\
u^{(2)}_\eps(t) &= S_\eps(t)u_0 + \int_0^t S_\eps(t-s)F(u^{(2)}_\eps(s))\,ds + \psi^\eps(t)\;,
\end{equs}
where the notations are the same as in \eref{e:defequ}.
We are then going to bound first the $L^\infty$-norm of $u^{(1)}_\eps - u$, then the $H^\alpha$-norm
of $u^{(1)}_\eps - u^{(2)}_\eps$ for some $\alpha > {1\over 2}$, and finally the $H^1$-norm of $u^{(2)}_\eps - u_\eps$.
Combining these bounds and using the Sobolev embedding $H^\alpha \hookrightarrow L^\infty$ then shows the required claim.
We now proceed to obtain these bounds. In the sequel, one should think of the terminal time $T$, the exponent $\kappa$,
and the (possibly random) `initial condition' $u_0 \in H^{3\over 2}$ as being fixed once and for all.

We also consider a (large) cutoff value $K>0$. One should think of $K$ as being a fixed constant
in the sequel although at the end we will take $K\to \infty$ as $\eps \to 0$.

\begin{notation}
From now on, we will denote by $C_K$ a generic constant that is allowed to depend only on the
particulars of the problem at hand ($f$, $g$, $h$, etc), on the choice of 
the cutoff value $K$, and possibly on $\kappa$. It is not allowed to depend on $\eps$ or on the initial condition $u_0$.
The precise value of $C_K$ is allowed to change without warning, even from one line to the next within the same equation.
\end{notation}

We define a first stopping time $\tau^K$ by 
\begin{equ}
\tau^K = \inf \{s \le T\,:\, \|u(s)\|_{L^\infty} \ge K\}\;,
\end{equ}
with the convention that $\tau^K = T$ if the set is empty.
Note that $\tau^K \to T \wedge \tau^*$ in probability as $K \to \infty$ by the definition of $\tau^*$. 

We then set 
\begin{equ}
\tau_1^K = \tau^K \wedge \inf \{s \le T\,:\, \|u(s) - u_\eps^{(1)}(s)\|_{L^\infty} \ge K\}\;.
\end{equ}
Our first step is to show that before the time $\tau_1^K$, the quantity $\|u(s) - u_\eps^{(1)}(s)\|_{L^\infty}$ is small with very high probability:

\begin{proposition}\label{prop:diff1}
For every $K>0$, one has
\begin{equ}
\lim_{\eps \to 0} \P \Bigl(\sup_{t \le \tau_1^K} \|u(t) - u_\eps^{(1)}(t)\|_{L^\infty} > \eps^{{1\over 2}-\kappa}\Bigr) = 0\;.
\end{equ}
\end{proposition}

\begin{proof}
Define the `residue' $R^{(1)}(t) = \psi^\eps(t) - \psi^0(t)$ and set $\rho(t) = u(t) - u_\eps^{(1)}(t) + R^{(1)}(t)$, so that
$\rho$ satisfies the integral equation
\begin{equ}
\rho(t) =  \int_0^t S(t-s)\bigl(F(u(s)) - F(u(s) - \rho(s) + R^{(1)}(s))\bigr)\,ds\;.
\end{equ}
Note now that for $s \le \tau_1^K$, the two arguments of $F$ are bounded by $2K$ in the $L^\infty$-norm 
by construction. Combining this with the local
Lipschitz property of $\bar f$ and the fact
that $S$ is a contraction semigroup in $L^\infty$, there exists
a constant $C_K$ such that 
\begin{equ}
\|\rho(t)\|_{L^\infty} \le  C_K\int_0^t \bigl(\|\rho(s)\|_{L^\infty} + \|R^{(1)}(s)\|_{L^\infty}\bigr)\,ds\;.
\end{equ}
The claim now follows at once from Proposition~\ref{prop:boundGauss}, combined with Gronwall's inequality.
\end{proof}

In our next step, we obtain a bound on the difference between $u^{(1)}_\eps$ and $u^{(1)}_\eps$ in a very similar way.
We define as before the stopping time 
\begin{equ}
\tau_2^K = \tau_1^K \wedge \inf \{s \le T\,:\, \|u_\eps^{(1)}(s) - u_\eps^{(2)}(s)\|_{\alpha} + \|\psi^\eps(s)\|_{L^\infty} \ge K\}\;,
\end{equ}
but this time it additionally depends on a parameter $\alpha$ to be determined later. 
One should think of $\alpha$ as being very close to, but slightly greater than, $1/2$. We then have,

\begin{proposition}\label{prop:diff2}
For every $\alpha \in ({1\over 2},1)$, one has
\begin{equ}
\lim_{\eps \to 0} \P \Bigl(\sup_{t \le \tau_2^K} \|u_\eps^{(1)}(t) - u_\eps^{(2)}(t)\|_{\alpha} > \eps^{{1\over 2}-\kappa}\Bigr) = 0\;.
\end{equ}
\end{proposition}

\begin{proof}
Writing $v_\eps^{(i)}(t) = u_\eps^{(i)}(t) - \psi^\eps(t)$, we have $v^{(i)}_\eps(0) = u_0$ and 
\begin{equs}
v^{(1)}_\eps(t) &= S(t-s)v^{(1)}_\eps(s) + \int_s^t S(t-r)F(v^{(1)}_\eps(r) + \psi^\eps(r))\,dr\;,\\
v^{(2)}_\eps(t) &= S_\eps(t-s)v^{(2)}_\eps(s) + \int_s^t S_\eps(t-r)F(v^{(2)}_\eps(r) + \psi^\eps(r))\,dr\;,
\end{equs}

As before, we define the residue $R^{(2)}(t) = S_\eps(t)u_0 - S(t)u_0$ and we note that
\begin{equ}
\sup_{s \le T} \|R^{(2)}(s)\|_\alpha \le C\eps^{1\over 2} \|u_0\|_{\alpha+{1\over 2}}\;,
\end{equ}
follows immediately from Lemma~\ref{lem:diffSu}.
Using the triangle inequality and the bound on the differences between the two semigroups given by
Lemma~\ref{lem:diffSu}, we furthermore obtain for any $u$, $v$ with $\|u\|_{L^\infty} + \|v\|_{L^\infty} \le K$
the bound
\begin{equs}
\|S(t) F(u) - S_\eps(t) F(v)\|_\alpha &\le \bigl\|\bigl(S(t) - S_\eps(t)\bigr)F(u)\bigr\|_\alpha + \bigl\|S_\eps(t) \bigl(F(u)-F(v)\bigr)\bigr\|_\alpha \\
&\le C_K \eps^{1\over 2} t^{-{\alpha \over 2} - {1\over 4}} + \bar C_K t^{-{\alpha \over 2}} \|u-v\|_{L^\infty}\;.
\end{equs}
In this bound, we single out the constant $\bar C_K$ in order to able to reuse its precise value later in the proof.
Setting $\rho = u_\eps^{(1)} - u_\eps^{(2)} = v_\eps^{(1)} - v_\eps^{(2)}$ we conclude that, for any $t \le \tau_2^K$, we have
\begin{equ}[e:boundrho1]
\|\rho(t)\|_\alpha \le C_K \eps^{1\over 2} (1 + \|u_0\|_{3\over 2}) + \bar C_K t^{1-{\alpha\over 2}} \sup_{0 < r < t} \|\rho(r)\|_\alpha\;.
\end{equ}
On the other hand, we obtain in a similar way from Lemma~\ref{lem:diffSu} the bound
\begin{equ}
\|S(t) u - S_\eps(t) v\|_\alpha \le \|u-v\|_\alpha + C \eps^{1\over 2} t^{-{\alpha \over 2} - {1\over 4}} \|v\|_{L^2}\;,
\end{equ}
so that
\begin{equ}[e:boundrho2]
\|\rho(t)\|_\alpha \le \|\rho(s)\|_\alpha + C_K \eps^{1\over 2} \bigl(1+(t-s)^{-{\alpha \over 2} - {1\over 4}}\bigr) + \bar C_K (t-s)^{1-{\alpha\over 2}} \sup_{s < r < t} \|\rho(r)\|_\alpha\;.
\end{equ}
In order to obtain a bound on $\rho$ over the whole time interval $[0, \tau_2^K]$, we now fix some $\delta_K$ sufficiently small so that 
$\bar C_K (2\delta_K)^{1-{\alpha\over 2}} \le {1\over 2}$. Setting
\begin{equ}
r_k \eqdef \sup_{k\delta_K \wedge \tau_2^K \le t \le (k+1)\delta_K \wedge \tau_2^K} \|\rho(t)\|_\alpha\;,
\end{equ}
it then follows from \eref{e:boundrho1} that
\begin{equ}[e:boundr0]
r_0 \vee r_1 \le C_K \eps^{1\over 2} (1 + \|u_0\|_{3\over 2})\;.
\end{equ}
Then, from \eref{e:boundrho2}, we obtain
\begin{equ}
r_{k+1} \le \|\rho\bigl((k-1)\delta_K \wedge \tau_2^K\bigr)\|_\alpha + C_K \eps^{1\over 2} \bigl(1+\delta_K^{-{\alpha \over 2} - {1\over 4}}\bigr) + {1\over 2} \bigl(r_k + r_{k+1}\bigr) \;,
\end{equ}
so that, since $\delta_K$ is independent of $\eps$, 
\begin{equ}[e:boundrk]
r_{k+1} \le 3 r_k + C_K \eps^{1\over 2} \;.
\end{equ}
Combining the bounds \eref{e:boundr0} and \eref{e:boundrk} shows that $\sup_{t \le \tau_2^K}\|\rho(t)\|_\alpha \le C_K \eps^{1/2}$, thus concluding the proof.
\end{proof}

It now remains to obtain a bound on the difference between $u_\eps^{(2)}(t)$ and $u_\eps$. This time, we are going to consider the 
$H^1$ norm, so that we set
\begin{equ}
\tau_3^K = \tau_1^K \wedge \inf \{s \le T\,:\, \|u_\eps^{(2)}(s) - u_\eps(s)\|_{1} \ge K\}\;.
\end{equ}

Note first that one has the following a priori bound:

\begin{lemma}\label{lem:aprioriueps}
The bound
\begin{equ}
\|u_\eps(t) - \psi^\eps(t)\|_1 \le C_K
\end{equ}
holds for every $t \le \tau_3^K$.
\end{lemma}

\begin{proof}
It follows from the definition of $u_\eps^{(2)}$ that
\begin{equ}
\|u_\eps^{(2)}(t) - \psi^\eps(t)\|_1 \le \|u_0\|_1 + \int_0^t \bigl\|S_\eps(t-s)F(u^{(2)}_\eps(s))\bigr\|_1\,ds\;.
\end{equ}
Since $\|F(u^{(2)}_\eps(s))\|_{L^\infty} \le C_K$ for $t \le \tau_3^K$, it follows from the regularising properties 
of the semigroup $S_\eps$ that
\begin{equ}
\|u_\eps^{(2)}(t) - \psi^\eps(t)\|_1 \le C_K\;.
\end{equ}
The claim then follows from the definitions of the various stopping times.
\end{proof}

\begin{proposition}\label{prop:diff3}
For every $\kappa > 0$ and every $K>0$, one has
\begin{equ}
\lim_{\eps \to 0} \P \Bigl(\sup_{t \le \tau_3^K} \|u_\eps^{(2)}(t) - u_\eps(t)\|_{1} > \eps^{{1\over 2}-\kappa}\Bigr) = 0\;.
\end{equ}
\end{proposition}

\begin{proof}
As before, we write $\rho = u_\eps - u_\eps^{(2)}$, so that
\begin{equ}
\rho(t) = \int_0^t S_\eps(t-s) \bigl(F_\eps(u_\eps(s)) - F(u_\eps(s)-\rho(s))\bigr)\,ds\;.
\end{equ}
Note now that since $\|u_\eps\|_{L^\infty}$ is bounded by a fixed constant before time $\tau_3^K$,
$\|F(u_\eps) - F(u_\eps - \rho)\|_{L^\infty}$ is bounded by a constant (depending only on $K$)
multiple of $\|\rho\|_{L^\infty}$ so that, for any two times $s< t \le \tau_3^K$, we have the almost sure bound
\begin{equs}
\|\rho(t)\|_1 &\le \|\rho(s)\|_1 + \Bigl\|\int_s^t S_\eps(t-r) \bigl(F_\eps(u_\eps(r)) - F(u_\eps(r))\bigr)\,dr\Bigr\|_1 \\
&\qquad + C_K (t-s)^{1\over 2} \sup_{s < r < t} \|\rho(r)\|_1\;.
\end{equs}
By breaking the interval $[0,T]$ into a finite number of subintervals $[t_k, t_{k+1}]$ of sufficiently short length, 
we conclude as above that
\begin{equ}
\sup_{t < \tau_3^K} \|\rho(t)\|_1 \le C_K \sup_k \sup_{t \in [t_k, t_{k+1}]} \Bigl\|\int_{t_k}^t S_\eps(t-s) \bigl(F_\eps(u_\eps(s)) - F(u_\eps(s))\bigr) \one_{s < \tau_3^K}\,ds\Bigr\|_1\;.
\end{equ}
In view of Proposition~\ref{prop:factor} below, it thus suffices to show that, for every $p>0$, there exists a constant $C_K$ depending also on $p$
such that
\begin{equ}[e:boundfdiff]
\E \bigl(\one_{s < \tau_3^K} \bigl\|F_\eps(u_\eps(s)) - F(u_\eps(s))\|_{-1}^p\bigr) \le C_K \eps^{{1 - \kappa \over 2} p}\;.
\end{equ}

In order to obtain such a bound, we write
\begin{equs}
F_\eps(u_\eps) &- F(u_\eps) = \eps \d_x \bigl(g(u_\eps)\d_x u_\eps\bigr) + \eps \bar h(u_\eps)\bigl(\bigl(\d_x u_\eps \otimes \d_x u_\eps\bigr) - A\bigr)\;,
\end{equs}
where we have set $\bar h = h - D_sg$ (see Remark~\ref{rem:notation} for the meaning of $D_s g$), and $A$ is a shortcut for ${1\over 2\sqrt \nu}$ times the identity matrix.

We then decompose $u_\eps$ as $u_\eps = v_\eps + \psi^\eps$
and recall that $\|v_\eps\|_1 \le C_K$ by Lemma~\ref{lem:aprioriueps}. 
Since, for times less than $\tau_3^K$, the $L^\infty$-norm of $g(u_\eps)$ and the $H^1$-norm of $v_\eps$ are 
almost surely bounded by a constant depending on $K$, it follows from the bounds of 
Lemma~\ref{lem:boundpsi} on the $H^1$-norm of $\psi^\eps$ that the first term does indeed
satisfy the bound \eref{e:boundfdiff}. The second term can be rewritten as
\begin{equs}
\eps \bar h(u_\eps)\bigl(\bigl(\d_x u_\eps \otimes &\d_x u_\eps\bigr) - A\bigr) = 
\eps \bar h(u_\eps)\bigl(\bigl(\d_x \psi^\eps \otimes \d_x \psi^\eps\bigr) - A\bigr) \\
&+ 2\eps \bar h(u_\eps)\bigl(\d_x v_\eps \otimes \d_x \psi^\eps\bigr) + \eps \bar h(u_\eps)\bigl(\d_x v_\eps \otimes \d_x v_\eps\bigr)\;.
\end{equs}
The terms on the second line satisfy the bound \eref{e:boundfdiff} in the same way as before, so that it remains to
consider the first term. Note that for times $t \le \tau_3^K$, one has the almost sure bound
\begin{equ}
\|\bar h(u_\eps(t))\|_\alpha \le C_K \bigl(1 + \|\psi^\eps(t)\|_\alpha\bigr)\;.
\end{equ}
This allows us to apply Corollary~\ref{cor:boundGauss}, thus showing that \eref{e:boundfdiff} also holds for this term
and concluding the proof.
\end{proof}

As an immediate corollary of propositions~\ref{prop:diff1}, \ref{prop:diff2}, and \ref{prop:diff3}, we have:
\begin{corollary}\label{cor:times}
One has $\lim_{\eps \to 0} \P(\tau^K = \tau_1^K = \tau_2^K = \tau_3^K) = 1$ for every $K>0$. In particular, 
for every choice of $K$, $\tau_3^K \to T \wedge \tau^*$ in probability as $\eps \to 0$.
\end{corollary}

\begin{proof}
It suffices to note that the additional \textit{a priori} constraint enforced by any of the $\tau_i^K$ is always 
given as $\rho^i(t) \le K$, for some quantity $\rho^i$ for which we then obtain a bound of the type 
\begin{equ}
\lim_{\eps \to 0} \P\Bigl( \sup_{t \le \tau_i^K} \rho^i(t) \ge \eps^\beta\Bigr) = 0\;,
\end{equ}
for some exponent $\beta > 0$. Since, setting $\tau_0^K = \tau^K$, we then have 
$\P(\tau_i^K \neq \tau_{i-1}^K) = \P\bigl( \rho^i(\tau_i^K) \ge K\bigr)$ for $i \in \{1,2,3\}$,
the claim follows at once.
\end{proof}

It is now straightforward to assemble all the pieces:

\begin{proof}[of Theorem~\ref{theo:mainrigor}]
Combining propositions~\ref{prop:diff1}--\ref{prop:diff3}, we obtain
\begin{equ}[e:finalconv]
\lim_{\eps \to 0} \P \Bigl(\sup_{t \le \tau_3^K} \|u_\eps(t) - u(t)\|_{L^\infty} > \eps^{{1\over 2}-\kappa}\Bigr) = 0\;.
\end{equ}
Since this is true for every $K>0$, a standard `diagonal' argument shows that 
there exists a sequence $K_\eps \to \infty$ such that \eref{e:finalconv}
still holds with $K$ replaced by $K_\eps$.
Since $\tau^K \to T \wedge \tau^*$ in probability as $K \to \infty$, it follows from Corollary~\ref{cor:times} that the
sequence of stopping times $\tau_\eps = \tau_3^{K_\eps}$ has the same property as $\eps \to0$, thus proving the claim.
\end{proof}

\subsection{Proof of Theorem~\ref{theo:convv}}
\label{sec:thm15}

The proof of Theorem~\ref{theo:convv} is slightly more straightforward than the proof of Theorem~\ref{theo:main},
 mainly due to the fact that we gain regularity in the limit $\eps \to 0$ instead of losing regularity.
First of all, we reformulate the solutions as solutions to the corresponding integral equations by setting:
\minilab{e:eqv}
\begin{equs}
v_\eps(t) &= S_\eps(t)v_0 + \int_0^t S_\eps(t-s)G(v_\eps(s))\,ds + \sqrt \eps \psi^\eps(t)\;,\label{e:epseqv} \\
v(t) &= S(t)v_0 + \int_0^t S(t-s)\bar G(v(s))\,ds\;,\label{e:maineqv} 
\end{equs}
where we have defined
\begin{equ}
G(u) = f(u) + h(u) \bigl(\d_x u \otimes \d_x u\bigr)\;,\qquad
\bar G(u) = \bar f(u) + h(u) \bigl(\d_x u \otimes \d_x u\bigr)\;,
\end{equ}
with $\bar f = f + {1\over 2\sqrt \nu} \Tr h$. This time, $v_0$ really is the initial condition for \eref{e:maineqv}, but as before
it is not the initial condition for \eref{e:epseqv} since $\psi^\eps(0) \neq 0$.
Again, we assume throughout that $f$ and $h$ are in $\CC^1$.
It then follows from Sobolev calculus and a standard bootstrapping argument that \eref{e:maineqv} is locally well-posed for 
initial conditions in $H^1$
and that its solutions, as long as they exist, belong to $H^s$ for all $s < 3$. Therefore, for every $v_0 \in H^1$,
there exists a (possibly infinite) time $\tau^*$ such that the solutions to \eref{e:maineqv} exist up to time $\tau^*$ and such that, 
provided $\tau^* \neq \infty$, one has $\lim_{t \to \tau^*} \|v(t)\|_1 = \infty$.

With this notation, we have the following rigorous formulation of Theorem~\ref{theo:convv}:

\begin{theorem}
Let $v_0$ be an $H^{2}$-valued random variable and fix $\beta \in ({1\over 2}, 1)$. Then, for every $\kappa > 0$ and every $T>0$, there exists a sequence of stopping times $\tau_\eps$ with $\tau_\eps \to T \wedge \tau^*$ in probability, so that
\begin{equ}
\lim_{\eps \to 0} \P \Bigl(\sup_{t\in [0,\tau_\eps]} \|v(t) - v_\eps(t)\|_{H^\beta} > \eps^{1-\beta - 2\kappa}\Bigr) = 0\;.
\end{equ}
\end{theorem}

\begin{remark}
Note that in this case, we are unable to deal with a nonlinearity of the form $g(v) \d_x^2 v$, mainly since this would
lead to a fully non-linear limiting equation, so that some of the techniques employed here break down.
\end{remark}

\begin{proof}
Fix two exponents $\alpha$, $\beta$ such that $\alpha \in (1, {3\over 2})$ and $\beta \in (2-\alpha, 1)$. One should think of $\alpha$ as being
slightly less than ${3\over 2}$ and $\beta$ as being slightly more than ${1\over 2}$. We also fix an exponent $\kappa < {1\over 4}$,
we set $\rho = v_\eps - v - \sqrt \eps \psi^\eps(t)$, and we define a stopping time $\tau^K$ by
\begin{equ}
\tau^K = \inf \{t < T\,:\, \eps^{\beta +\kappa - {1\over 2}}\|\psi^\eps(t)\|_\beta + \sqrt \eps \|\psi^\eps(t)\|_{1-\kappa} + \|\rho(t)\|_\alpha + \|v(t)\|_2 > K\}\;.
\end{equ}
For $t \le \tau^K$, one then has the bound
\begin{equs}
\|\rho(t)\|_\alpha &\le C\eps^{1\over 2} \|v_0\|_2 + \Bigl\|\int_0^t S_\eps(t-s) \bigl(G\bigl(v_\eps(s)\bigr) - \bar G\bigl(v(s)\bigr)\bigr)\Bigr\|_\alpha\,ds \\
&\qquad +  \int_0^t \bigl\|\bigl(S_\eps(t-s) - S(t-s)\bigr)\bar G(v(s))\bigr\|_\alpha\,ds\;, \\
&\le \eps^{1\over 2} \bigl(C\|v_0\|_2 + C_K\bigr) + \Bigl\|\int_0^t S_\eps(t-s) \bigl(G\bigl(v_\eps(s)\bigr) - \bar G\bigl(v(s)\bigr)\bigr)\Bigr\|_\alpha\,ds\;,
\end{equs}
where we have made use of Lemma~\ref{lem:diffSu} and the fact that $\|v(s)\|_2 < K$ in order to bound the last term. Similarly, we obtain
for $s < t < \tau^K$ the bound
\begin{equs}
\|\rho(t)\|_\alpha \le \|\rho(s)\|_\alpha + C_K \eps^{1\over 2} + \Bigl\|\int_s^t S_\eps(t-r) \bigl(G\bigl(v_\eps(r)\bigr) - \bar G\bigl(v(r)\bigr)\bigr)\Bigr\|_\alpha\,dr\;.
\end{equs}
Before we proceed, we obtain bounds on the last term in this inequality.

For this, we rewrite $G(v_\eps) - \bar G(v)$ as 
\begin{equs}
G(v_\eps) - \bar G(v) &= \bigl(f(v_\eps) - f(v)\bigr) + \bigl(h(v_\eps) - h(v)\bigr) \bigl(\d_x v \otimes \d_x v + A\bigr) \\
&\quad + h(v_\eps) \bigl(\eps \d_x \psi^\eps \otimes \d_x \psi^\eps - A\bigr) \\
&\quad + h(v_\eps) \bigl(\d_x \rho \otimes (2\d_x v + \d_x \rho) \bigr)\\
&\quad + 2\sqrt \eps h(v_\eps) \bigl(\d_x \psi^\eps \otimes \d_x v)\\
&\quad + 2\sqrt \eps h(v_\eps) \bigl(\d_x \psi^\eps \otimes \d_x \rho\bigr)\\
&\eqdef I_1 + I_2 +I_3 + I_4 + I_5\;,
\end{equs}
and we bound these terms separately. Since it follows from the definition of $\tau^K$ that both $v$ and $v_\eps$ are bounded 
by some constant $C_K$ in 
$H^{1-\kappa}$ (and therefore in particular in $L^\infty$), it follows that
\minilab{e:bounds}
\begin{equ}[e:boundI1]
\|I_1\|_{L^2} \le C_K \bigl(\eps^{1-\beta-\kappa} + \|\rho\|_{L^2}\bigr)\;.
\end{equ}
In order to bound $I_2$, we use the fact that $\|h(v_\eps)\|_{\beta} \le C_K$ almost surely to apply Corollary~\ref{cor:conc}.
Therefore, for every $\gamma > {1\over 2}$ and every $p > 0$, we obtain constants $C_K(p)$ such that
\minilab{e:bounds}
\begin{equ}[e:boundI2]
\E \one_{t\le \tau^K} \|I_2(t)\|_{-\gamma}^p \le C_K(p) \eps^{p \over 2}\;.
\end{equ}
In order to bound $I_3$, we use the fact that both $\d_x \rho$ and $\d_x v$ are bounded in $L^2$,
together with the fact that $L^1 \subset H^{-\gamma}$ for every $\gamma > {1\over 2}$, so that we have
\minilab{e:bounds}
\begin{equ}[e:boundI3]
\|I_3\|_{-\gamma} \le C_K \|\rho\|_\alpha\;.
\end{equ}
The bound on $I_4$ exploits the fact that $\sqrt \eps \d_x\psi^\eps$ is bounded by $K \eps^{1-\beta-\kappa}$ in 
$H^{\beta-1}$, together with the fact that $h(v_\eps) \d_x v$ is bounded by a constant in $H^{1-\kappa}$. Furthermore,
the product of an element in  $H^{\beta-1}$ with an element in $H^{1-\kappa}$ lies in $H^{-{1\over 2}}$, provided that 
$\kappa < \beta$ (which is always the case since we assumed that $\beta > {1\over 2}$), so that
\begin{equ}
\|I_4\|_{-{1\over 2}} \le C_K \eps^{1-\beta-\kappa}\;.
\end{equ}
Finally, in order to bound $I_5$, we use the fact that $\sqrt \eps \d_x\psi^\eps$ is bounded by $K$ in 
$H^{-\kappa}$, $h(v_\eps) \d_x \rho$ is bounded by $C_K \|\rho\|_\alpha$ in $H^{{1\over 2}-\kappa}$,
and $\kappa < {1\over 4}$ to deduce that 
\begin{equ}
\|I_5\|_{-{1\over 2}} \le C_K \|\rho\|_\alpha\;.
\end{equ} 
Making sure that the constant $\gamma$ appearing in the bounds on $I_2$ and $I_3$ is less that $2-\alpha$,
we can combine all of these bounds in order to get the inequality
\begin{equs}
\|\rho(t\wedge \tau^K)\|_\alpha &\le \|\rho(s\wedge \tau^K)\|_\alpha + R_s(t) + C_K |t-s|^\delta \sup_{s \le r \le t} \|\rho(r\wedge \tau^K)\|_\alpha\;,
\end{equs}
where the remainder term $R_s$ satisfies
\begin{equ}
\E \sup_{t \le \tau^K} \|R_s(t)\|_{\alpha} \le C_K \bigl(\eps^{1-\beta - {3\over 2}\kappa}+ \eps^{1\over 2}\|v_0\|_2\bigr)\;,
\end{equ}
by Proposition~\ref{prop:factor}. We now conclude in the same way as in the proof of Proposition~\ref{prop:diff3}.
\end{proof}

\appendix

\section{Factorisation method}

The aim of this section is to show the following:

\begin{proposition}\label{prop:factor}
For every $\kappa > 0$ there exists $p > 0$ such that the bound
\begin{equ}
\E \sup_{t \in [0,T]} \Bigl\| \int_0^t S_\eps(t-s)\,R(s)\,ds \Bigr\|_{\gamma}^p \le C \eps^{-\kappa p}  \sup_{t \in [0,T]} \E \|R(t)\|_{\gamma-2}^p\;,
\end{equ}
holds for every $\gamma \in \R$ and every $H^{\gamma-2}$-valued stochastic process $R$.
\end{proposition}

\begin{proof}
The proof relies on the `factorisation method' as in \cite{DKZ87}.
We make use of the following bounds. Recall Minkowski's integral inequality,
\begin{equ}[e:Young]
\biggl(\E \Bigl(\int_0^t f(s)\,ds\Bigr)^p\biggr)^{1/p} \le \int_0^t \bigl(\E f^p(s)\bigr)^{1/p}\,ds\;,
\end{equ}
which holds for every $p \ge 1$ and every real-valued stochastic process $f$. Given $R$, we then define a process $y$ by
\begin{equ}[e:defy]
y(t) = \int_0^t (t-s)^{-\alpha} S_\eps(t-s) R(s)\,ds\;,
\end{equ}
where $\alpha$ is some (typically very small) exponent to be determined later. The convolution of $R$ with $S_\eps$ is then given
by
\begin{equ}[e:bound]
\tilde R(t) \eqdef \int_0^t S_\eps(t-s)\,R(s)\,ds = C_\alpha \int_0^t (t-s)^{\alpha-1} S_\eps(t-s) \,y(s)\,ds\;,
\end{equ}
where the constant $C_\alpha$ gets large as $\alpha$ gets small. It is known that for every $p > 1/\alpha$, 
the map $y \mapsto \tilde R$ given by \eref{e:bound} is bounded from $L^p([0,T],H^\gamma)$ into $\CC([0,T],H^\gamma)$ \cite{Hairer08}, so that
\begin{equ}
\E \sup_{t \in [0,T]} \|\tilde R(t) \|_\gamma^p \le C \sup_{t \in [0,T]}\E \|y(t)\|_\gamma^p\;.
\end{equ}
It thus suffices to show that $\E \|y(s)\|_\gamma^p$ is bounded 
by $\eps^{-\kappa p} \sup_{t \in [0,T]} \E \|R(t)\|_{\gamma-2}^p$ for some large enough value of $p$.

Combining \eref{e:defy} with \eref{e:Young}, we have the bound
\begin{equ}
\bigl(\E \|y(t)\|_\gamma^p\bigr)^{1/p} \le \int_0^t (t-s)^{-\alpha} \bigl(\E \|S_\eps(t-s) R(s)\|_\gamma^p\bigr)^{1/p}\,ds\;.
\end{equ}
It then follows from Lemma~\ref{lem:smoothSeps} that there exists a constant $C$ such that
\begin{equ}
\bigl(\E \|y(t)\|_\gamma^p\bigr)^{1/p} \le C \eps^{-4\alpha} \int_0^t (t-s)^{\alpha-1} \bigl(\E \|R(s)\|_{\gamma-2} ^p\bigr)^{1/p}\,ds\;,
\end{equ}
so that the required bound follows by setting $\alpha = \kappa/4$.
\end{proof}

\endappendix

\bibliographystyle{Martin}
\bibliography{./refs}

\def\cprime{$'$}
\begin{thebibliography}{HSVW05}
\expandafter\ifx\csname url\endcsname\relax
  \def\url#1{\texttt{#1}}\fi
\expandafter\ifx\csname urlprefix\endcsname\relax\def\urlprefix{URL }\fi

\bibitem[Ass02]{MR1888875}
\textsc{S.~Assing}.
\newblock A pregenerator for {B}urgers equation forced by conservative noise.
\newblock \emph{Comm. Math. Phys.} \textbf{225}, no.~3, (2002), 611--632.

\bibitem[ASZ09]{ASZ}
\textsc{L.~Ambrosio}, \textsc{G.~Savar{\'e}}, and \textsc{L.~Zambotti}.
\newblock Existence and stability for {F}okker-{P}lanck equations with
  log-concave reference measure.
\newblock \emph{Probab. Theory Related Fields} \textbf{145}, no. 3-4, (2009),
  517--564.

\bibitem[AvR10]{Kostja}
\textsc{S.~Andres} and \textsc{M.-K. von Renesse}.
\newblock Particle approximation of the {W}asserstein diffusion.
\newblock \emph{J. Funct. Anal.} \textbf{258}, no.~11, (2010), 3879--3905.

\bibitem[BG97]{MR1462228}
\textsc{L.~Bertini} and \textsc{G.~Giacomin}.
\newblock Stochastic {B}urgers and {KPZ} equations from particle systems.
\newblock \emph{Comm. Math. Phys.} \textbf{183}, no.~3, (1997), 571--607.

\bibitem[Bor75]{Borell}
\textsc{C.~Borell}.
\newblock The {B}runn-{M}inkowski inequality in {G}auss space.
\newblock \emph{Invent. Math.} \textbf{30}, no.~2, (1975), 207--216.

\bibitem[Cha00]{MR1743612}
\textsc{T.~Chan}.
\newblock Scaling limits of {W}ick ordered {KPZ} equation.
\newblock \emph{Comm. Math. Phys.} \textbf{209}, no.~3, (2000), 671--690.

\bibitem[DGF75]{GammaConv}
\textsc{E.~De~Giorgi} and \textsc{T.~Franzoni}.
\newblock Su un tipo di convergenza variazionale.
\newblock \emph{Atti Accad. Naz. Lincei Rend. Cl. Sci. Fis. Mat. Natur. (8)}
  \textbf{58}, no.~6, (1975), 842--850.

\bibitem[DPD03]{MR2016604}
\textsc{G.~Da~Prato} and \textsc{A.~Debussche}.
\newblock Strong solutions to the stochastic quantization equations.
\newblock \emph{Ann. Probab.} \textbf{31}, no.~4, (2003), 1900--1916.

\bibitem[DPDT07]{MR2365646}
\textsc{G.~Da~Prato}, \textsc{A.~Debussche}, and \textsc{L.~Tubaro}.
\newblock A modified {K}ardar-{P}arisi-{Z}hang model.
\newblock \emph{Electron. Comm. Probab.} \textbf{12}, (2007), 442--453
  (electronic).

\bibitem[DPKZ87]{DKZ87}
\textsc{G.~Da~Prato}, \textsc{S.~Kwapie{\'n}}, and \textsc{J.~Zabczyk}.
\newblock Regularity of solutions of linear stochastic equations in {H}ilbert
  spaces.
\newblock \emph{Stochastics} \textbf{23}, no.~1, (1987), 1--23.

\bibitem[DPT07]{DaPT}
\textsc{G.~Da~Prato} and \textsc{L.~Tubaro}.
\newblock Wick powers in stochastic pdes: an introduction, 2007.
\newblock Preprint.

\bibitem[DPZ92]{DaPrato-Zabczyk92}
\textsc{G.~Da~Prato} and \textsc{J.~Zabczyk}.
\newblock \emph{Stochastic Equations in Infinite Dimensions}, vol.~44 of
  \emph{Encyclopedia of Mathematics and its Applications}.
\newblock Cambridge University Press, 1992.

\bibitem[Fer74]{MR0413237}
\textsc{X.~Fernique}.
\newblock Des r\'esultats nouveaux sur les processus gaussiens.
\newblock \emph{C. R. Acad. Sci. Paris S\'er. A} \textbf{278}, (1974),
  363--365.

\bibitem[Fre04]{MR2099730}
\textsc{M.~Freidlin}.
\newblock Some remarks on the {S}moluchowski-{K}ramers approximation.
\newblock \emph{J. Statist. Phys.} \textbf{117}, no. 3-4, (2004), 617--634.

\bibitem[Hai09]{Hairer08}
\textsc{M.~Hairer}.
\newblock An introduction to stochastic {PDEs}.
\newblock Lecture notes, 2009.
\newblock \urlprefix\url{http://arxiv.org/abs/0907.4178}.

\bibitem[H{\O}UZ96]{WickBook}
\textsc{H.~Holden}, \textsc{B.~{\O}ksendal}, \textsc{J.~Ub{\o}e}, and
  \textsc{T.~Zhang}.
\newblock \emph{Stochastic partial differential equations}.
\newblock Probability and its Applications. Birkh\"auser Boston Inc., Boston,
  MA, 1996.
\newblock A modeling, white noise functional approach.

\bibitem[HSV07]{HairerStuartVoss07}
\textsc{M.~Hairer}, \textsc{A.~M. Stuart}, and \textsc{J.~Voss}.
\newblock Analysis of {SPDEs} arising in path sampling, part~{II}: The
  nonlinear case.
\newblock \emph{Annals of Applied Probability} \textbf{17}, no. 5/6, (2007),
  1657--1706.

\bibitem[HSV09]{Hypo}
\textsc{M.~Hairer}, \textsc{A.~M. Stuart}, and \textsc{J.~Vo{\ss}}.
\newblock Sampling conditioned hypoelliptic diffusions, 2009.
\newblock Preprint.

\bibitem[HSVW05]{HairerStuartVossWiberg05}
\textsc{M.~Hairer}, \textsc{A.~M. Stuart}, \textsc{J.~Voss}, and
  \textsc{P.~Wiberg}.
\newblock Analysis of {SPDEs} arising in path sampling, part~{I}: The
  {G}aussian case.
\newblock \emph{Communications in Mathematical Sciences} \textbf{3}, no.~4,
  (2005), 587--603.

\bibitem[JLM85]{MR815192}
\textsc{G.~Jona-Lasinio} and \textsc{P.~K. Mitter}.
\newblock On the stochastic quantization of field theory.
\newblock \emph{Comm. Math. Phys.} \textbf{101}, no.~3, (1985), 409--436.

\bibitem[Kol06]{Kol}
\textsc{A.~V. Kolesnikov}.
\newblock Mosco convergence of {D}irichlet forms in infinite dimensions with
  changing reference measures.
\newblock \emph{J. Funct. Anal.} \textbf{230}, no.~2, (2006), 382--418.

\bibitem[KP92]{MR1214374}
\textsc{P.~E. Kloeden} and \textsc{E.~Platen}.
\newblock \emph{Numerical solution of stochastic differential equations},
  vol.~23 of \emph{Applications of Mathematics (New York)}.
\newblock Springer-Verlag, Berlin, 1992.

\bibitem[KPS04]{MR2130323}
\textsc{R.~Kupferman}, \textsc{G.~A. Pavliotis}, and \textsc{A.~M. Stuart}.
\newblock It\^o versus {S}tratonovich white-noise limits for systems with
  inertia and colored multiplicative noise.
\newblock \emph{Phys. Rev. E (3)} \textbf{70}, no.~3, (2004), 036120, 9.

\bibitem[KS03]{KuShi}
\textsc{K.~Kuwae} and \textsc{T.~Shioya}.
\newblock Convergence of spectral structures: a functional analytic theory and
  its applications to spectral geometry.
\newblock \emph{Comm. Anal. Geom.} \textbf{11}, no.~4, (2003), 599--673.

\bibitem[Mos67]{MoscoOld}
\textsc{U.~Mosco}.
\newblock Approximation of the solutions of some variational inequalities.
\newblock \emph{Ann. Scuola Norm. Sup. Pisa (3) 21 (1967), 373--394; erratum,
  ibid. (3)} \textbf{21}, (1967), 765.

\bibitem[Mos94]{Mosco}
\textsc{U.~Mosco}.
\newblock Composite media and asymptotic {D}irichlet forms.
\newblock \emph{J. Funct. Anal.} \textbf{123}, no.~2, (1994), 368--421.

\bibitem[MR10]{Boris}
\textsc{R.~Mikulevicius} and \textsc{B.~L. Rozovskii}.
\newblock On quantization of stochastic {N}avier-{S}tokes equation, 2010.
\newblock Preprint.

\bibitem[Pug09]{pugachev}
\textsc{O.~V. Pugachev}.
\newblock On mosco convergence of diffusion dirichlet forms.
\newblock \emph{Theory of Probability and its Applications} \textbf{53}, no.~2,
  (2009), 242--255.

\bibitem[RY91]{MR1083357}
\textsc{D.~Revuz} and \textsc{M.~Yor}.
\newblock \emph{Continuous martingales and {B}rownian motion}, vol. 293 of
  \emph{Grundlehren der Mathematischen Wissenschaften [Fundamental Principles
  of Mathematical Sciences]}.
\newblock Springer-Verlag, Berlin, 1991.

\bibitem[SC74]{SudTsi}
\textsc{V.~N. Sudakov} and \textsc{B.~S. Cirel{\cprime}son}.
\newblock Extremal properties of half-spaces for spherically invariant
  measures.
\newblock \emph{Zap. Nau\v cn. Sem. Leningrad. Otdel. Mat. Inst. Steklov.
  (LOMI)} \textbf{41}, (1974), 14--24, 165.
\newblock Problems in the theory of probability distributions, II.

\bibitem[Tal87]{MR906527}
\textsc{M.~Talagrand}.
\newblock Regularity of {G}aussian processes.
\newblock \emph{Acta Math.} \textbf{159}, no. 1-2, (1987), 99--149.

\bibitem[Tal95]{Talagrand}
\textsc{M.~Talagrand}.
\newblock Concentration of measure and isoperimetric inequalities in product
  spaces.
\newblock \emph{Inst. Hautes \'Etudes Sci. Publ. Math.} , no.~81, (1995),
  73--205.

\bibitem[WZ65]{MR0183023}
\textsc{E.~Wong} and \textsc{M.~Zakai}.
\newblock On the relation between ordinary and stochastic differential
  equations.
\newblock \emph{Internat. J. Engrg. Sci.} \textbf{3}, (1965), 213--229.

\end{thebibliography}

\end{document}